\documentclass[10pt,a4paper]{article}
\usepackage[left=3cm,right=3cm,top=1.5cm,bottom=1.6cm]{geometry}
\usepackage[utf8]{inputenc}
\usepackage[english]{babel}
\usepackage[T1]{fontenc}
\usepackage{amsmath}
\usepackage{amsthm}
\usepackage{amsfonts}
\usepackage{amssymb}
\usepackage{dsfont}
\usepackage{graphicx}
\usepackage[lofdepth,lotdepth]{subfig}
\usepackage[usenames,dvipsnames,svgnames,table]{xcolor}
\usepackage{kpfonts}
\usepackage{extarrows}
\usepackage[colorlinks,citecolor=blue,linkcolor=blue]{hyperref}
\usepackage{enumitem}

\DeclareMathOperator{\B}{\mathbb{B}}
\DeclareMathOperator{\R}{\mathbb{R}}

\DeclareMathOperator{\Z}{\mathbb{Z}}
\DeclareMathOperator{\ESP}{\mathbb{E}}
\DeclareMathOperator{\Prob}{\mathbb{P}}

\def\Ld{\mathrm{L}^{2}}

\def\o{\omega}

\def\ga{\gamma}

\def\ii{+\infty}
\def\ind{\mathds{1}}
\def\unda{\underline{a}}
\def\ova{\overline{a}}

\newtheorem{thm}{Theorem}[section]
\newtheorem{lemma}{Lemma}[section]
\newtheorem{proposition}{Proposition}[section]

\newtheorem{definition}{Definition}[section]
\newtheorem{rem}{\textbf{Remark}}[section]

\title{A new Multifractional Process with Random Exponent}
\author{Antoine Ayache \\
Univ. Lille, CNRS, UMR 8524 - Laboratoire Paul Painlevé,\\ F-59000 Lille, France.\\
E-mail: \texttt{antoine.ayache@math.univ-lille1.fr}\\
\and
C\'eline Esser\\
Universit\'e de Li\`ege - Institut de Math\'ematique,\\ All\'ee de la D\'ecouverte 12, 4000 Li\`ege, Belgium\\
Email: \texttt{celine.esser@uliege.be}
\and
Julien Hamonier \\
Univ. Lille, CHU Lille, EA 2694 - Santé publique : \\ épidémiologie et qualité des soins, F-59000 Lille, France.\\
E-mail: \texttt{julien.hamonier@univ-lille2.fr}
 }
\date{}

\date{}

\begin{document}
\maketitle

\begin{abstract}
A first type of Multifractional Process with Random Exponent (MPRE) was constructed several years ago in \cite{ayache2005multifractional} by replacing in a wavelet series representation of Fractional Brownian Motion (FBM) the Hurst parameter by a random variable depending on the time variable. In the present article, we propose another approach for constructing another type of MPRE. It consists in substituting to the Hurst parameter, in a stochastic integral representation of the high-frequency part of FBM, a random variable depending on the integration variable. The MPRE obtained in this way offers, among other things, the advantages to have a representation through classical It\^o integral and to be less difficult to simulate than the first type of MPRE, previously introduced in \cite{ayache2005multifractional}. Yet, the study of H\"older regularity of this new MPRE is a significantly more challenging problem than in the case of the previous one. Actually, it requires to develop a new methodology relying on an extensive use of the Haar basis.
\end{abstract}

 \vspace{4ex}

\noindent {\bf MSC2010:}  60G17, 60G22\\
\noindent{\bf Keywords:} Fractional Brownian Motion, varying Hurst parameter, Haar basis, H\"older regularity, It\^o integral

\section{Introduction}
\label{sec:Intro}

Fractional Brownian Motion (FBM), which was introduced by Kolmogorov \cite{kolmogorov:1940} and made popular by Mandelbrot and Van Ness  
\cite{mandelbrot:van_ness:1968}, is one of the most important stochastic process in both theory and applications (see for instance \cite{samorodnitsky:taqqu:1994book,doukhan:oppenheim:taqqu:2002-livre}). This continuous centred Gaussian process $\{B_H(t): t\in I\}$, where $I$ denotes the closed interval $[0,1]$, depends on a deterministic constant parameter, denoted by $H$, belonging to the open interval $(0,1)$ and called the Hurst parameter. Let 
$\{B(s): s\in \R\}$ be a standard Brownian Motion on a complete probability space $(\Omega,\mathcal{F},\Prob)$ and let $\int (\cdot) dB$ be the associated Wiener integral. Then the FBM  $\{B_H(t): t\in I\}$ can be defined, for every $t\in I$, as:
\begin{equation}\label{def_fbm}
B_H(t):=\int_{-\infty}^0 \Big ((t-s)^{H-\frac{1}{2}}-(-s)^{H-\frac{1}{2}}\Big)dB(s)+\int_{0}^{t} (t-s)^{H-\frac{1}{2}}dB(s). 
\end{equation}
The first one of these two Wiener integrals is called the low-frequency part of FBM and the other one its high-frequency part. Roughness of paths of FBM is mainly due to its high-frequency part, which is also called the Riemann-Liouville process and denoted by $\{R_H (t):t\in I\}$. For the sake or clarity, let us point out that, for all $t\in I$, one has
\begin{equation}\label{def_RiemLiou}
R_H (t):=\int_{0}^{t} (t-s)^{H-\frac{1}{2}}dB(s). 
\end{equation}
Roughness of paths of $\{R_H (t):t\in I\}$, and consequently that of $\{B_H(t): t\in I\}$, is governed by the Hurst parameter $H$. More precisely, on any arbitrary non-degenerate compact interval included in $I$, the critical H\"older regularity of paths of these two processes is equal to $H$. Thus, in contrast with many real-life signals, roughness of paths of FBM is not allowed to change from one period of time to another which somehow restricts its areas of applicability. 

In spite of importance and usefulness of FBM as a random model in signal processing, the constancy and non-randomness of its Hurst parameter $H$ are serious limitations of it. This is the main motivation
behind construction and study of several classes of Multifractional Processes since the mid-1990s (see for instance \cite{ayache2005multifractional, ayache2007wavelet, Bi, benassi97,peltier95, stoev:06, surgailis:08} to mention just a few references). The paradigmatic example of such processes  is the Gaussian Multifractional Brownian Motion (MBM) of Benassi, Jaffard, L\'evy V\'ehel, Peltier and Roux \cite{benassi97,peltier95}, which is obtained simply by replacing in (\ref{def_fbm}) the constant Hurst parameter $H$ by a deterministic function $H(t)$, depending on the time variable $t$ in a continuous way. Observe that the assumption that $H(t)$ is deterministic, or more generally that the processes $\{H(t):t\in I\}$ and $\{B(t):t\in I\}$ are independent, is crucial. Indeed, as explained in \cite{ayache2005multifractional}, the stochastic integrals in (\ref{def_fbm}) fail to be well-defined, at least in the usual It\^o sense, when $H$ is replaced by a stochastic process $\{S(t): t\in I\}$ which is not independent on the Brownian Motion $\{B(s): s\in \R\}$. In order to overcome this difficulty, the article \cite{ayache2005multifractional} proposed to replace $H$ by $\{S(t): t\in I\}$ in an almost surely uniformly convergent random wavelet series representation of FBM, due to Meyer, Sellan and Taqqu \cite{meyer:sellan:taqqu:1999}, which is of a rather different nature from its stochastic integral representation (\ref{def_fbm}). Thus, the article \cite{ayache2005multifractional} was able to construct a first type of Multifractional Process with Random Exponent (MPRE) which, among other things, turned out to be useful in stock prices modelling, thanks to some papers by Bianchi and his co-authors (see for instance \cite{BiPP1,BiPP2,BiP1}).

The main goal of our present article is to propose another approach for constructing a new type of MPRE which, among other things, offers the advantages to have a representation through classical It\^o integral and to be less difficult to simulate than the first type of MPRE, previously introduced in \cite{ayache2005multifractional}. Yet, as we will see it in our article, the study of H\"older regularity of this new MPRE is a significantly more challenging problem than in the case of the previous one. Actually, it requires to develop a new methodology relying on an extensive use of the Haar basis \cite{haar,Dau92,Meyer92,Meyer90,wojt}, whose definition is recalled in (\ref{A:eq:haar}) in the next section, and which is sometimes called in French in a humorous way: "l'ondelette du pauvre" (the wavelet of the poor). The approach we propose in our present article is to a certain extent inspired by the one used by Surgailis in \cite{surgailis:08} which, roughly speaking, consists in replacing the constant Hurst parameter $H$ by a deterministic function $H(s)$ depending on the integration variable $s$ (and not on the time variable $t$). Yet, in our case, $H(s)$ is not only a deterministic function but more generally a stochastic process with continuous paths, denoted by  $\left\lbrace A(s): s\in I\right\rbrace$, which is assumed to be adapted to the natural filtration $(\mathcal{F}_s)_{s\in I}$ associated with the Brownian Motion $\{B(s):s\in I\}$; recall that $\mathcal{F}_s:=\sigma\left( B(u);0\le u \le s\right)$, for all $s\in I$. Another important assumption on $\left\lbrace A(s): s\in I\right\rbrace$
%
%
is that it takes its values in $[\unda,\ova] $ where $\unda$ and $\ova$ are two deterministic arbitrary constants satisfying the inequalities:
\begin{equation}\label{eqn:bornesA}
\frac{1}{2} < \unda \le \ova <1 \, .
\end{equation}

The MPRE we study in our present article is denoted by $\lbrace X(t):t\in I \rbrace$ and obtained by substituting in (\ref{def_RiemLiou}) the process $\left\lbrace A(s): s\in I\right\rbrace$  to the Hurst parameter $H$. More precisely, $\lbrace X(t):t\in I \rbrace$ is defined, for all $t\in I$, as the It\^o integral:
\begin{equation}\label{def_X}
X(t)=\int_0^1 K_t(s) dB(s),
\end{equation}
where, for any $(t,s)\in I^2$,
\begin{equation}\label{def_Kt}
K_t(s):= (t-s)_+^{A(s)-\frac{1}{2}}=
\left\lbrace
\begin{array}{ll}
0 & \mbox{ if }t\le s\\
(t-s)^{A(s)-\frac{1}{2}} & \mbox{ otherwise.}
\end{array}
\right.
\end{equation}
One clearly has, for any $s\in I$, that $K_0(s)=0$ and therefore $X(0)=0$. Let us show that, for any fixed $t\in (0,1]$, the It\^o integral in~(\ref{def_X}), makes sense; that is the stochastic process $\{K_t (s):s\in I\}$ belongs to usual class of integrands for the It\^o integral over $I$.
\begin{enumerate}
\item From Relation~(\ref{def_Kt}), for any fixed $\o\in  \Omega$, the function $s\mapsto K_t(s,\o)$ is continuous on $I$. Therefore $(s,\o) \mapsto K_t(s,\o)$ is a measurable function from $(I\times \Omega,\B(I)\otimes\mathcal{F})$ to $(\R,\B(\R))$. 
\item For any $s\in I$, the random variable $K_t(s)$ is $\mathcal{F}_s$-measurable; this is obvious in view of (\ref{def_Kt}), since $A(s)$ is a $\mathcal{F}_s$-measurable random variable.
\item One has $K_t \in \Ld(I\times \Omega,\B(I)\otimes \mathcal{F})$; indeed, it follows from (\ref{def_Kt}) and the fact $A(s)\in [\unda,\ova]\subset (1/2,1)$, that 
\begin{eqnarray}
\int_{\Omega}\left(\int_0^1 |K_t(s,\o)|^2 ds \right) d\Prob &=& \int_{\Omega}\left(\int_0^t (t-s)^{2A(s)-1} ds \right) d\Prob \nonumber\\
&\le & \int_{\Omega}\left(\int_0^t (t-s)^{2\unda-1} ds \right) d\Prob <\ii\,.
\end{eqnarray}
\end{enumerate}

From now on, we assume, in addition, that there exist a constant $\rho \in (0,1]$ and a positive constant $c$ such that, for any $x,y\in I$, one has
\begin{equation}\label{hyp:Ahold1}
\ESP\big(|A(x)-A(y)|^2\big) \le c |x-y|^{2\rho}\,.
\end{equation}

Under this additional assumption, the stochastic process $\{X(t):t\in I\}$ has a modification whose paths are H\"older continuous function on $I$. More precisely, the following proposition holds.

\begin{proposition}\label{prop:holdun}
The process $\{X(t) : t\in I\}$ has a modification whose paths are H\"older continuous functions on $I$ of any order $\zeta\in (0,\unda-1/2)$. 
\end{proposition}

Let us emphasize that throughout our article the process $\{X(t):t\in I\}$ is systematically identified with its modification with H\"older continuous paths introduced in Proposition~\ref{prop:holdun}. 

Also, it is worth mentioning that the proof of Proposition~\ref{prop:holdun}, which is given in the Appendix~\ref{sec:Appen}, mainly relies on the classical Kolmogorov-\v{C}entsov's continuity theorem (see \cite{karatzas2012brownian} for instance). It is well-known that the latter theorem is of simple use and of very great utility. However, it can hardly allow to obtain very precise results on H\"older regularity of non-Gaussian processes such as the MPRE $\{X(t) : t\in I\}$. Thus, one of the main goals of our present article is to derive, thanks to a different methodology relying on the use of the Haar basis, a much more precise result than Proposition~\ref{prop:holdun} on H\"older regularity of $\{X(t) : t\in I\}$, namely Theorem~\ref{thm:holder}. 


The rest of our article is organized in the following way. In Section~\ref{sec:repX}, we introduce via the Haar basis a random series representation of the MPRE $\{X(t):t\in I\}$, and we derive the almost sure convergence of the series uniformly in $t\in I$. In Section~\ref{sec:Hol}, we assume in addition that the paths of $\{A(s): s\in I\}$ satisfy a uniform H\"older condition of an arbitrary order $\ga>1/2$, and under this additional assumption, we show that the pathwise uniform H\"older exponent of $\{X(t):t\in I\}$ on any arbitrary interval $[\nu_1,\nu_2]\subseteq I$ is almost surely bounded from below by $\min_{s\in [\nu_1,\nu_2]} A(s)$. This result somehow means that roughness of paths of $\{X(t):t\in I\}$ is governed by  $\{A(s):s\in I\}$ and thus is allowed to change from one period of time to another; in order to derive it, we make an essential use of the series representation of $\{X(t):t\in I\}$ via the Haar basis. In Section~\ref{sec:Sim}, thanks to the latter representation of $\{X(t):t\in I\}$, we give two simulation methods for this MPRE and we test them; our simulations tend to confirm the fact that roughness of paths of $\{X(t):t\in I\}$ does not remain everywhere the same and is closely connected to the values of $\{A(s):s\in I\}$. In the Appendix~\ref{sec:Appen} the proofs of some auxiliary results are given.

\section{Series representation via the Haar basis}
\label{sec:repX}

First, we recall that the Haar basis of $L^2([0,1])$ is the collection of functions: 
\begin{equation}
\label{A:eq:haar}
\left\{
\begin{array}{l}
\mathcal{U} := \ind_{[0,1)} \\
h_{j,k} := 2^{\frac{j}{2}} \left(\ind_{\big[ 2^{-j}k, 2^{-j}(k+\frac{1}{2})\big)} - \ind_{\big[ 2^{-j}(k+\frac{1}{2}), 2^{-j}(k+1)\big)} \right), \quad  j \in \Z_+\mbox{ and } k \in \{ 0, \dots, 2^j-1\}.
\end{array}
\right.
\end{equation}
Let us point out that, for each $j\in\Z_+$ and $k\in\{ 0, \dots, 2^j\}$, the dyadic number $k2^{-j}$ is frequently denoted by $\delta_{j,k}$. Moreover, when $k<2^j$, we frequently denote by $\Delta B_{j,k}$ the increment of the Brownian motion $\{B(s):s\in I\}$ on the dyadic interval $[ \delta_{j,k}, \delta_{j,k+1}) $, that is one has
\begin{equation}\label{def:delta}
\Delta B_{j,k} := B(\delta_{j,k+1})-B(\delta_{j,k}).
\end{equation}
The main result of the present section is the following theorem which gives a random series representation of the MPRE $\lbrace X(t):t\in I \rbrace$. We mention that, roughly speaking, this representation of $\lbrace X(t):t\in I \rbrace$ is obtained via the decomposition of the associated kernel function $K_t(\cdot,\omega)$ in the Haar basis. 

\begin{thm}\label{thm:Haarrepres}
Assume that the exponent $\rho$ (see (\ref{hyp:Ahold1})) satisfies 
\begin{equation}\label{cond2:rho}
 \frac{1}{2}<\rho \le 1\,.
\end{equation}
Then, there is an event $\Omega_{**} \subseteq \Omega$ of probability 1 such that, for every $\omega \in \Omega_{**}$, one has
\begin{equation}\label{eq:Haarrepres}
X(t,\omega) = \big\langle K_t(\cdot,\omega), \mathcal{U} \big\rangle \eta_0(\omega) + \sum_{j=0}^{+ \infty} \sum_{k=0}^{2^j-1} \big\langle K_t(\cdot,\omega), h_{j,k}\big\rangle \varepsilon_{j,k}(\omega) \, ,
\end{equation}
where the convergence holds uniformly in $t \in I$, and where the $\mathcal{N}(0,1)$ Gaussian random variables $\eta_0$ and $\varepsilon_{j,k}$, $j \in \Z_+$, $k \in \{0,\dots, 2^j-1\}$ are defined by
\begin{equation}\label{eq:eta}
\eta_0 := \int_0^1 \mathcal{U}(s) dB(s) = B(1) - B(0) = \Delta B_{0,0}
\end{equation}
and
\begin{align}\label{eq:epsilon}
\varepsilon_{j,k} := \int_0^1 h_{j,k}(s) dB(s) & = 2^{\frac{j}{2}} \Bigg( 2 B\left( 2^{-(j+1)}(2k+1) \right) - B\left( 2^{-j}k \right) -  B\left( 2^{-j}(k+1) \right) \Bigg) \nonumber \\
& = 2^{\frac{j}{2}} \left(\Delta B_{j+1,2k} - \Delta B_{j+1,2k+1} \right) \, .
\end{align}
\end{thm}

Our first goal is to show that the convergence in (\ref{eq:Haarrepres}) holds, for each fixed $t\in I$, in $L^1(\Omega)$. To this end, we need some preliminary results. 

\begin{lemma}\label{lem_accroissK}
There exists a deterministic constant $c_0>0$ such that, for every real numbers $s',s'',t$ satisfying
\begin{equation}\label{eq:cond:lem_accroissK}
0 \leq s' \leq s'' < t \leq 1,
\end{equation}
one has
\begin{equation}\label{eq:lem_accroissK}
\big|K_t(s') - K_t(s'') \big| \leq c_0 \left( (t-s'')^{\underline{a}- \frac{3}{2} }(s''-s') + \big|A(s') - A(s'') \big| \right).
\end{equation}
\end{lemma}

The proof of Lemma~\ref{lem_accroissK} is given in the Appendix~\ref{sec:Appen}.

\begin{rem}
Let us mention that, for all $t \in I$, for every $j \in \Z_+$, and for each $k \in \{0,\dots,2^j-1\}$, one has
\begin{align}\label{eq:rem:Kt:hjk}
\big\langle K_t, h_{j,k}\big\rangle  = & 2^{\frac{j}{2}}  \int_0^1 K_t(s)  \left(\ind_{\big[ 2^{-j}k, 2^{-j}(k+\frac{1}{2})\big)} - \ind_{\big[ 2^{-j}(k+\frac{1}{2}), 2^{-j}(k+1)\big)} \right) ds  \nonumber \\
 = & 2^{\frac{j}{2}}   \left( \int_{2^{-j}k}^{2^{-j}(k+\frac{1}{2})} K_t(s)  ds  - \int_{2^{-j}(k+\frac{1}{2})}^{2^{-j}(k+1)} K_t(s) \, ds \right) \nonumber \\
 = & 2^{\frac{j}{2}}   \int_{2^{-j}k}^{2^{-j}(k+\frac{1}{2})} \big(K_t(s) - K_t(s+2^{-j-1}) \big) \,  ds  \, .
\end{align}
Therefore, using the inequalities $0 \leq K_t(s,\omega) \leq 1$ one gets that
\begin{equation}\label{eq:rem:Kt:hjk2}
\left| \big\langle K_t, h_{j,k}\big\rangle \right| \leq 2^{-\frac{j}{2}} \, .
\end{equation}
\end{rem}

The partial sums of the series in (\ref{eq:Haarrepres}) are defined in the following way:

\begin{definition}
For all $t\in I$ and for each $J\in\Z_+$, one sets
\begin{equation}\label{eqn:XJt2}
X^J(t):= \big\langle K_t, \mathcal{U} \big\rangle \eta_0 + \sum_{j=0}^{J-1} \sum_{k=0}^{2^j-1} \big\langle K_t, h_{j,k}\big\rangle \varepsilon_{j,k} \, ,
\end{equation}
with the convention that when $J=0$ one has
\[
X^0 (t):= \big\langle K_t, \mathcal{U} \big\rangle \eta_0.
\]
\end{definition}

The following proposition provides an alternative expression of the process $\{ X^J(t) : t \in I \}$. 

\begin{proposition}\label{prop:mvKt}
For every $J\in\Z_+$ and for all $l\in \{0,\ldots,2^J-1\}$, let $\overline{K}_t^{J,l}$ be the mean value of the function $s\mapsto K_t(s)$ on the dyadic interval $[\delta_{J,l},\delta_{J,l+1}]$ defined as:
\begin{equation}\label{eqn:mvKt}
\overline{K}_t^{J,l}:=2^J\int_{\delta_{J,l}}^{\delta_{J,l+1}} K_t(s) ds.
\end{equation}
Then, for every $J\in\Z_+$ and for all $l\in \{0,\ldots,2^J-1\}$, one has
\begin{equation}
\label{eq:alignement1}
0\le\overline{K}_t^{J,l}\le 1\quad\quad\mbox{and}\quad\quad \overline{K}_t^{J,l}=\dfrac{\overline{K}_t^{J+1,2l}+\overline{K}_t^{J+1,2l+1}}{2}\,.
\end{equation}
Moreover, for all $t\in I$ and for each $J\in\Z_+$, $X^J(t)$ can be expressed as:
\begin{equation}\label{eqn:XJt3}
X^J(t)= \sum_{l=0}^{2^J-1} \overline{K}_t^{J,l} \Delta B_{J,l} \, ,
\end{equation}
where $\Delta B_{J,l}$ is the increment of the Brownian motion $B$ defined through (\ref{def:delta}) with $j=J$ and $k=l$.

\end{proposition}

\begin{proof}
The proof of (\ref{eq:alignement1}) is skipped since it is very easy. For proving (\ref{eqn:XJt3}) one proceeds by induction on $J$. It is clearly satisfied when $J=0$. Let us assume that it holds for an arbitrary $J \in \Z_+$ and show that it remains true when $J$ is replaced by $J+1$. Thus, in view of the induction hypothesis  and of (\ref{eqn:XJt2}), it suffices to prove that
\begin{equation}\label{eqn:mvKt:proof1}
\sum_{k=0}^{2^J-1} \big\langle K_t, h_{J,k}\big\rangle \varepsilon_{J,k} =  \sum_{l=0}^{2^{J+1}-1} \overline{K}_t^{J+1,l} \Delta B_{J+1,l} - \sum_{l=0}^{2^J-1} \overline{K}_t^{J,l} \Delta B_{J,l}.
\end{equation}
First, let us note that
\begin{align*}
\big\langle K_t, h_{J,k}\big\rangle & = 2^{\frac{J}{2}} \left( \int_{\delta_{J+1,2k}}^{\delta_{J+1,2k+1}}K_t(s) ds - \int_{\delta_{J+1,2k+1}}^{\delta_{J+1,2k+2}}K_t(s) ds\right)  \nonumber \\
& = 2^{-\frac{J}{2}-1} \left( \overline{K}_t^{J+1,2k} - \overline{K}_t^{J+1,2k+1} \right)\,.
\end{align*}
Thus, one can derive from (\ref{eq:epsilon}) that
\begin{align}\label{eqn:mvKt:proof2}
\sum_{k=0}^{2^J-1} \big\langle K_t, h_{J,k}\big\rangle \varepsilon_{J,k} & = \sum_{k=0}^{2^J-1
} \left( \frac{\overline{K}_t^{J+1,2k} - \overline{K}_t^{J+1,2k+1}}{2} \right) \left( \Delta B_{J+1,2k} - \Delta B_{J+1,2k+1}\right).
\end{align}
On the other hand, using the equality 
$
\Delta B_{J,l} = \Delta B_{J+1, 2l} + \Delta B_{J+1,2l+1} 
$
and the equality in (\ref{eq:alignement1}) one obtains that
\begin{align}\label{eqn:mvKt:proof3}
& \sum_{l=0}^{2^{J+1}-1} \overline{K}_t^{J+1,l} \Delta B_{J+1,l}  - \sum_{l=0}^{2^J-1} \overline{K}_t^{J,l} \Delta B_{J,l} \nonumber \\
& = \sum_{k=0}^{2^{J}-1} \left( \overline{K}_t^{J+1,2k} \Delta B_{J+1,2k} +  \overline{K}_t^{J+1,2k+1} \Delta B_{J+1,2k+1} \right)  \nonumber \\ 
&\hspace{2cm}-\sum_{k=0}^{2^J-1}\left (\dfrac{\overline{K}_t^{J+1,2k}+\overline{K}_t^{J+1,2k+1}}{2} \right)\left( \Delta B_{J+1, 2k} + \Delta B_{J+1,2k+1} \right) \nonumber \\
& =  \sum_{k=0}^{2^{J}-1}\left( \left( \dfrac{\overline{K}_t^{J+1,2k}-\overline{K}_t^{J+1,2k+1}}{2} \right) \Delta B_{J+1,2k} - \left( \dfrac{\overline{K}_t^{J+1,2k}-\overline{K}_t^{J+1,2k+1}}{2} \right)  \Delta B_{J+1,2k+1} \right) \nonumber \\
& = \sum_{k=0}^{2^J-1
} \left( \frac{\overline{K}_t^{J+1,2k} - \overline{K}_t^{J+1,2k+1}}{2} \right) \left( \Delta B_{J+1,2k} - \Delta B_{J+1,2k+1}\right) \,.
\end{align}
Thus combining (\ref{eqn:mvKt:proof3}) and (\ref{eqn:mvKt:proof2}) one gets (\ref{eqn:mvKt:proof1}).
\end{proof}

In order to show that, for each fixed $t\in I$, $X^J (t)$ converges to $X(t)$ in $L^1 (\Omega)$ when $J$ goes to $+\infty$, let us introduce the stochastic process $\big\{\widetilde{X}^J(t): t\in I\big\}$ defined through It\^o integral in the following way: 

\begin{definition}\label{def:KJt}
For all fixed $J\in\Z_+$ and $t\in I$, let $\big\{\widetilde{K}^J_t(s) : s\in I\big\}$ be the elementary stochastic process defined, for every $s\in I$, as: 
\begin{equation}\label{eqn:KJt}
\widetilde{K}^J_t(s):=\sum_{l=0}^{2^J-1} K_t(\delta_{J,l}) \ind_{[\delta_{J,l},\delta_{J,l+1})}(s).
\end{equation}
One sets
\begin{equation}\label{eqn:XJt}
\widetilde{X}^J(t):=\int_0^1 \widetilde{K}_t^J(s) dB(s)=\sum_{l=0}^{2^J-1}K_t(\delta_{J,l}) \Delta B_{J,l}.
\end{equation}
\end{definition}


\begin{rem}\label{rem:L}
For each fixed $t\in I$, the function $L_t$ is defined, for all $(u,v)\in [0,1]\times [\underline{a},\overline{a}]$, as:
\begin{equation}\label{eq:Ldef}
L_t(u,v):= (t-u)_+^{v-\frac{1}{2}}=
\left\lbrace
\begin{array}{ll}
0 & \mbox{ if }t\le u \, , \\
(t-u)^{v-\frac{1}{2}} & \mbox{ otherwise.}
\end{array}
\right.
\end{equation}
This function is continuous on $[0,1]\times [\underline{a},\overline{a}]$ and satisfies, for every $(u,v) \in [0,1]\times [\underline{a},\overline{a}]$,
\begin{equation}\label{eq:Lencadr}
0 \leq L_t(u,v) \leq 1 \, .
\end{equation}
Moreover, in view of (\ref{def_Kt}) and (\ref{eq:Ldef}) one has 
\begin{equation}
\label{eq:antF_KL}
K_t (s,\o)=L_t (s,A(s,\o))\,,\quad \mbox{for all $(t,s,\o)\in I^2\times\Omega$,}
\end{equation}
\end{rem}

Remark \ref{rem:L}, the continuity of the paths of the process $\{A(s) : s \in I\}$, 
the dominated convergence theorem, and the isometry property of It\^o integral easily imply that the following lemma holds. 

\begin{lemma}\label{lem:convXJt}
For any fixed $(t,s,\o)\in I^2\times \Omega$, one has
\begin{equation}
\lim_{J\rightarrow \ii} \widetilde{K}_t^J(s,\o)=K_t(s,\o).
\end{equation}
Hence, using the dominated convergence theorem and the isometry property of It\^o integral, one gets:
\begin{enumerate}
\item the sequence $(\widetilde{K}^J_t)_{J\in\Z_+}$ converges to $K_t$ in $\Ld(I\times \Omega)$ when $J$ goes to $\ii$;
\item the sequence $\left( \widetilde{X}^J(t) \right)_{J\in\Z_+}$ converges to $X(t)= \int_0^1 K_t(s) dB(s)$ in $\Ld(\Omega)$ when $J$ goes to $\ii$.
\end{enumerate}
\end{lemma}

\begin{lemma}\label{lem:conv:XJtXt}
Assume that the exponent $\rho$ (see (\ref{hyp:Ahold1})) satisfies  (\ref{cond2:rho}). Then, one has
\begin{equation}\label{eqn:conv:XJtXt}
\lim_{J\rightarrow \ii} \left\lbrace \sup_{t\in I} \ESP\left( \left| X^J(t)-\widetilde{X}^J(t)  \right| \right) \right\rbrace=0.
\end{equation}
\end{lemma}

The proof of Lemma~\ref{lem:conv:XJtXt} is given in the Appendix~\ref{sec:Appen}. The following proposition is a straightforward consequence of Lemma \ref{lem:convXJt}, Lemma \ref{lem:conv:XJtXt} and (\ref{eqn:XJt2}). 

\begin{proposition}\label{prop:Haarrepres}
Assume that the exponent $\rho$ (see (\ref{hyp:Ahold1})) satisfies  (\ref{cond2:rho}). Then, for every fixed $t \in  I$, one has
\begin{equation}\label{eq:Haarrepres2}
X(t) = \big\langle K_t, \mathcal{U} \big\rangle \eta_0 + \sum_{j=0}^{+ \infty} \sum_{k=0}^{2^j-1} \big\langle K_t, h_{j,k}\big\rangle \varepsilon_{j,k} \, ,
\end{equation}
where the convergence holds in $L^1(\Omega)$. 
\end{proposition}


\begin{proposition}\label{prop:unif}
Assume that the exponent $\rho$ (see (\ref{hyp:Ahold1})) satisfies  (\ref{cond2:rho}). Then one has
\begin{equation}\label{eqn:lem:unif0}
\ESP \left( \sum_{j=0}^{+ \infty} \sup_{t \in I} \left( \sum_{k=0}^{2^j-1} \left| \big\langle K_t, h_{j,k}\big\rangle \right| |\varepsilon_{j,k}| \right)\right) < + \infty .
\end{equation}
\end{proposition}

One of the main ingredients of the proof of Proposition~\ref{prop:unif} is the following lemma which concerns the ${\cal N}(0,1)$ Gaussian random variables $\varepsilon_{j,k}$ and whose proof can be found in \cite{ayache2003rate}.

\begin{lemma}\label{lem:classic}
There are an event $\Omega_\ast \subseteq \Omega$ of probability 1 and a non-negative random variable $C_\ast$ with finite moment of any order, such that the inequality
\begin{equation}
|\varepsilon_{j,k}(\omega)| \leq C_\ast(\omega) \sqrt{j+1}.
\end{equation}
holds, for all $\omega \in \Omega_\ast$, for every $j \in \Z_+$ and for each $k \in \{0,\dots, 2^j -1\}$.  
\end{lemma}

\begin{proof}[Proof of Proposition~\ref{prop:unif}]
For all $t \in I$ and all $j \in \Z_+$, one has
\begin{equation}
\label{eq:zeroparent}
\sum_{k=0}^{2^j-1} \left| \big\langle K_t, h_{j,k}\big\rangle \right| =\left| \big\langle K_t, h_{j,[2^jt]}\big\rangle \right| +\sum_{k=0}^{[2^jt]-1} \left| \big\langle K_t, h_{j,k}\big\rangle \right| \, ,
\end{equation}
with the conventions that $\big\langle K_t, h_{j,2^j}\big\rangle=0$ and $\sum_{k=0}^{-1}\cdots=0$, which means that the second term in the right-hand side of (\ref{eq:zeroparent}) vanishes when $[2^jt]=0$.
In the other case $[2^jt]\ge 1$, using (\ref{eq:rem:Kt:hjk}) and Lemma~\ref{lem_accroissK}, one gets
\begin{align*}
\sum_{k=0}^{[2^jt]-1} \left| \big\langle K_t, h_{j,k}\big\rangle \right| 
& \le c_0 2 ^{\frac{j}{2}} \sum_{k=0}^{[2^jt]-1} \int_{2^{-j}k}^{2^{-j}(k + \frac{1}{2})} \left( (t-2^{-j-1}-s)^{\underline{a}- \frac{3}{2}} 2^{-j-1} + \big| A(s) - A(s+2^{-j-1})\big| \right)ds \nonumber \\
& \le  c_0 2^{-\frac{j}{2}-1}  \int_{0}^{t-2^{-j-1}}  (t-2^{-j-1}-s)^{\underline{a}- \frac{3}{2}}  ds +  c_0 2^{\frac{j}{2}}  \int_0^{t-2^{-j-1}} \big| A(s) - A(s+2^{-j-1})\big| ds \nonumber \\
& \le c_0 2^{-\frac{j}{2}-1}   \frac{(t-2^{-j-1})^{\underline{a}- \frac{1}{2}}}{\underline{a}- \frac{1}{2}}   +  c_0 2 ^{\frac{j}{2}}  \int_0^{t-2^{-j-1}} \big| A(s) - A(s+2^{-j-1})\big| ds \nonumber \\
& \le c_3 2^{-\frac{j}{2}}   +  c_0 2 ^{\frac{j}{2}}  \int_0^{t-2^{-j-1}} \big| A(s) - A(s+2^{-j-1})\big| ds
\end{align*}
where $c_3 = \frac{c_0}{2\underline{a}-1}$. Moreover, from (\ref{eq:rem:Kt:hjk2}), we know that, for any $t\in I$,
\begin{equation*}
\left| \big\langle K_t, h_{j,[2^jt]}\big\rangle \right| \leq 2^{-\frac{j}{2}}.
\end{equation*}
Thus, setting $c_4 := c_3+1$, one has, for each $t\in I$,
\begin{equation}\label{eqn:lem:unif}
\sum_{k=0}^{2^j-1} \left| \big\langle K_t, h_{j,k}\big\rangle \right| \leq c_4 2^{-\frac{j}{2}}   +  c_0 2 ^{\frac{j}{2}}  \int_0^{1-2^{-j-1}} \big| A(s) - A(s+2^{-j-1})\big| ds \, .
\end{equation}
Then using Lemma~\ref{lem:classic} and (\ref{eqn:lem:unif}), one obtains, for every $\omega \in \Omega_\ast$ and for all $j \in \Z_+$, that
\begin{equation*}
\sup_{t \in I} \left( \sum_{k=0}^{2^j-1} \left| \big\langle K_t, h_{j,k}\big\rangle \right| |\varepsilon_{j,k}| \right)
\le C_{\ast} \sqrt{1+j} \left( c_4 2^{-\frac{j}{2}}   +  c_0 2^{\frac{j}{2}}  \int_0^{1-2^{-j-1}} \big| A(s) - A(s+2^{-j-1})\big| ds \right),
\end{equation*}
where  the non-negative random variable $C_\ast$ has finite moments of any order. Thus, in order to derive (\ref{eqn:lem:unif0}), it is enough to prove that 
\begin{equation}\label{eqn:lem:unif1}
\sum_{j=0}^{+ \infty} 2^{\frac{j}{2}} \sqrt{1+j}\,\ESP \left( C_{\ast} \int_0^{1-2^{-j-1}} \big| A(s) - A(s+2^{-j-1})\big| ds \right) < + \infty \, .
\end{equation}
For every $j \in \Z_+$, it follows from the Cauchy-Schwarz inequality and (\ref{hyp:Ahold1}) that
\begin{align}\label{eqn:lem:unif2}
&\ESP \left( C_{\ast} \int_0^{1-2^{-j-1}} \big| A(s) - A(s+2^{-j-1})\big| ds \right) \nonumber\\ 
& \le  \left( \ESP \left( C_{\ast}^2 \right) \right)^{\frac{1}{2}}  \left( \ESP \left( \left( \int_0^{1-2^{-j-1}} \big| A(s) - A(s+2^{-j-1})\big| ds  \right)^2 \right) \right)^{\frac{1}{2}}  \nonumber \\
& \le \left( \ESP \left( C_{\ast}^2 \right) \right)^{\frac{1}{2}}  \left( \ESP  \left( \int_0^{1-2^{-j-1}} \big| A(s) - A(s+2^{-j-1})\big|^2 ds  \right) \right)^{\frac{1}{2}}  \nonumber \\
& = \left( \ESP \left( C_{\ast}^2 \right) \right)^{\frac{1}{2}}    \left( \int_0^{1-2^{-j-1}}\ESP \left(\big| A(s) - A(s+2^{-j-1})\big|^2 \right) ds  \right)^{\frac{1}{2}}  \nonumber \\
& \leq c_5 2^{-\rho j}
\end{align}
where $c_5>0$ is a deterministic finite constant not depending on $j$ and $t$. Finally, combining  (\ref{eqn:lem:unif2}) and the assumption (\ref{cond2:rho}), one gets (\ref{eqn:lem:unif1}). 
\end{proof}

\begin{proof}[Proof of Theorem \ref{thm:Haarrepres}] The theorem is a straightforward consequence of
Propositions \ref{prop:Haarrepres} and \ref{prop:unif}. 
\end{proof}

\section{Study of H\"older regularity}
\label{sec:Hol}


\begin{definition}
Let $f$ be a deterministic real-valued continuous function defined on $I$. The uniform H\"older exponent of $f$ on an arbitrary non-degenerate compact interval $[\nu_1, \nu_2] \subseteq I$ is denoted by $\beta_f\big([\nu_1, \nu_2] \big)$ and defined as:
\begin{equation}
\beta_f\big([\nu_1, \nu_2] \big) := \sup \left\{ \gamma \in [0,1] : \sup_{t',t'' \in [\nu_1, \nu_2]} \frac{\big| f(t'') - f(t')\big|}{|t''-t'|^\gamma} < + \infty \right\} \, . 
\end{equation}
\end{definition}

Throughout this section, we assume that there are a deterministic constant $\gamma  \in ( \frac{1}{2},1)$ and $\Omega_0$ an event of probability 1, such that, for each $\omega\in\Omega_0$, the path $A(\cdot,\omega)$ of the process $\{A(s) : s \in I\}$ satisfies, on a $I$, a uniform H\"older condition of order $\gamma$. That is there exists a finite constant $C_1(\omega)$ such that, for all $s',s''\in I$, the following inequality holds :
\begin{equation}\label{eqn:regholdA}
\left| A(s'',\omega) - A(s',\omega) \right| \leq C_1(\omega) \left| s''-s' \right|^\gamma \, .
\end{equation}

\begin{thm}\label{thm:holder}
Assume that the condition (\ref{eqn:regholdA}) is satisfied. Then, there exists an event $\Omega_1 \subseteq \Omega$ of probability 1 such that, for every $\omega \in \Omega_1$ and for all non-degenerate compact interval $[\nu_1,\nu_2] \subseteq I$, one has
\begin{equation}
\label{eq:A:thm:holder}
\beta_{X(\cdot, \omega)}\big( [\nu_1,\nu_2] \big) \geq A_{\nu_1, \nu_2}(\omega)\, ,
\end{equation}
where
\begin{equation}\label{thm:holder:def}
A_{\nu_1, \nu_2}(\omega) := \min \big\{A(s,\omega) : s \in [\nu_1, \nu_2] \big\} \, .
\end{equation}
\end{thm}

\begin{rem}
We conjecture that the inequality in (\ref{eq:A:thm:holder}) can be replaced by an equality. 
\end{rem}

\begin{proof}
Let us consider an arbitrary non-degenerate compact interval $[\nu_1,\nu_2] \subseteq I$, and let us fix $t',t'' \in [\nu_1,\nu_2]$ such that $t'<t''$. One can derive from Theorem \ref{thm:Haarrepres} and the triangular inequality that, on some event $\Omega_{**}$ of probability~1 included in $\Omega_\ast$ (see Lemma~\ref{lem:classic}) and not depending on $t',t'',\nu_1,\nu_2$, the following inequality holds: 
\begin{equation}
\label{A:thm3.1:eq1}
\left| X(t'')-X(t')\right| \leq \left|\big\langle K_{t''}, \mathcal{U} \big\rangle - \big\langle K_{t'}, \mathcal{U} \big\rangle \right| |\eta_0| + \sum_{j=0}^{+ \infty} \sum_{k=0}^{2^j-1} \left| \big\langle K_{t''}, h_{j,k}\big\rangle -\big\langle K_{t'}, h_{j,k}\big\rangle \right| | \varepsilon_{j,k}| \, ,
\end{equation}
Observe that it can easily be seen that, for all fixed $\omega\in\Omega$ (the probability space), $t\mapsto \big\langle K_{t}(\omega), \mathcal{U} \big\rangle$ is a Lipschitz function on $I$. So, we only have to focus on the second term in the right-hand side of (\ref{A:thm3.1:eq1}).
Using Lemma \ref{lem:classic}, one has
\begin{equation}\label{eqn0:thm:holder}
\sum_{j=0}^{+ \infty} \sum_{k=0}^{2^j-1} \left| \big\langle K_{t''}, h_{j,k}\big\rangle -\big\langle K_{t'}, h_{j,k}\big\rangle \right| | \varepsilon_{j,k}| \leq C_\ast \sum_{j=0}^{+ \infty} \left( \sqrt{j+1} \sum_{k=0}^{2^j-1} \left| \big\langle K_{t''}, h_{j,k}\big\rangle -\big\langle K_{t'}, h_{j,k}\big\rangle \right| \right).
\end{equation}
Moreover, in view of (\ref{eq:rem:Kt:hjk}) and (\ref{eq:antF_KL}), for each $j \in \Z_+$, one has
\begin{align}
 \sum_{k=0}^{2^j-1}\left| \big\langle K_{t''}, h_{j,k}\big\rangle -  \big\langle K_{t'}, h_{j,k}\big\rangle  \right|
 \leq  2^{\frac{j}{2}} & \sum_{k=0}^{2^j-1} \int_{2^{-j}k}^{2^{-j}(k + \frac{1}{2})}  \Big|L_{t''}\big(s,A(s) \big) - L_{t''}\big(s+2^{-j-1}, A(s+2^{-j-1})\big)   \nonumber \\ 
 & \qquad \qquad \qquad - L_{t'}\big(s,A(s)\big) + L_{t'}\big(s+2^{-j-1},A(s+2^{-j-1})\big) \Big| ds \nonumber \\
 \leq  2^{\frac{j}{2}} & \int_0^{1-2^{-j-1}}  \Big| L_{t''}\big(s,A(s) \big) - L_{t''}\big(s+2^{-j-1}, A(s+2^{-j-1})\big)   \nonumber \\ 
 & \qquad \qquad  - L_{t'}\big(s,A(s)\big) + L_{t'}\big(s+2^{-j-1},A(s+2^{-j-1})\big)  \Big| ds  \, .
\end{align}
Therefore, using the triangular inequality, one gets that
\begin{equation}\label{eqn:thm:holder}
\sum_{k=0}^{2^j-1} \left| \big\langle K_{t''}, h_{j,k}\big\rangle -  \big\langle K_{t'}, h_{j,k}\big\rangle \right| \leq \lambda^1_{j}(t',t'') + \lambda^2_{j}(t',t'')\,,
\end{equation}
where 
\begin{align}\label{eqn1:thm:holder}
\lambda^1_{j}(t',t'') := 2^{\frac{j}{2}}  \int_0^{1-2^{-j-1}}   & \Big| L_{t''}\big(s,A(s) \big) - L_{t''}\big(s, A(s+2^{-j-1})\big)  \nonumber \\
& - L_{t'}\big(s,A(s)\big) + L_{t'}\big(s,A(s+2^{-j-1})\big) \Big| ds
\end{align}
and
\begin{align}\label{eqn2:thm:holder}
 \lambda^2_{j}(t',t'') 
 :=  2^{\frac{j}{2}}  \int_0^{1-2^{-j-1}} & \Big| L_{t''}\big(s, A(s+2^{-j-1})\big) -L_{t''}\big(s+2^{-j-1}, A(s+2^{-j-1})\big) \nonumber \\
& - L_{t'}\big(s,A(s+2^{-j-1})\big)  + L_{t'}\big(s+2^{-j-1},A(s+2^{-j-1})\big) \Big|  ds \, .
\end{align}
Let us set
\begin{equation}\label{eqn3:thm:holder}
t'_j := \min \{t', 1-2^{-j-1}\} \quad \text{and} \quad t''_j := \min \{t'', 1-2^{-j-1}\} \, ,
\end{equation}
and let $j_0$ denote the unique nonnegative integer such that
\begin{equation}\label{eqn224:thm:holder}
2^{-j_0-1} < t''-t' \leq 2^{-j_0} \, .
\end{equation}
From now on and till the end of the proof, we work on the event  $\Omega_1 := \Omega_{**} \cap \Omega_0$ of probability 1.
\bigskip

\textbf{$\bullet$ Step 1 : Upper bound for $\lambda^1_{j}(t',t'') \, $. } \\
From (\ref{eq:Ldef}),  (\ref{eqn1:thm:holder}) and (\ref{eqn3:thm:holder}), we know that
\begin{align}\label{eqn100:thm:holder}
 \lambda^1_{j}(t',t'')  = & 2^{\frac{j}{2}} \int_0^{t'_j}  \Big| \big(t''-s\big)^{A(s) -\frac{1}{2}} - \big(t''-s\big)^{A(s+2^{-j-1})-\frac{1}{2}}   - \big(t'-s\big)^{A(s)-\frac{1}{2}} + \big(t'-s\big)^{A(s+2^{-j-1}) - \frac{1}{2}} \Big| ds \nonumber \\
&   \qquad  + 2^{\frac{j}{2}} \int_{t'_j}^{t''_j}\Big| \big(t''-s\big)^{A(s) - \frac{1}{2} } - \big(t''-s\big)^{A(s+2^{-j-1})-\frac{1}{2}}   \Big| ds \, .
\end{align}
Let us first conveniently bound from above the first integral in (\ref{eqn100:thm:holder}). Observe that this integral vanishes when $t'_j=0$, so there is no restriction to assume that  $t'_j>0$. Let us then fix $s \in (0,t'_j)$. By applying the Mean Value Theorem to the function $x \mapsto (t''-s)^{x-\frac{1}{2}} - (t'-s)^{x-\frac{1}{2}}$, one gets
\begin{align}\label{eqn10:thm:holder}
& \Big| \big(t''-s\big)^{A(s) -\frac{1}{2}} - \big(t''-s\big)^{A(s+2^{-j-1})-\frac{1}{2}}   - \big(t'-s\big)^{A(s)-\frac{1}{2}} + \big(t'-s\big)^{A(s+2^{-j-1}) - \frac{1}{2}} \Big| \nonumber \\
  = &  \left| (t''-s)^{e_1 - \frac{1}{2}} \log (t''-s) - (t'-s)^{e_1 - \frac{1}{2}} \log (t'-s) \right| \big|A(s) - A(s+2^{-j-1})\big|
\end{align}
where $e_1$ depends on $t',t'',s,j$ and is between $A(s)$ and $A(s+2^{-j-1})$. This implies that
\begin{equation}\label{eqn11:thm:holder}
e_1 \in [\underline{a}, \overline{a} ] \subseteq \big( \frac{1}{2},1 \big).
\end{equation}
Let us apply once again the Mean Value Theorem to the function $x \mapsto x^{e_1 - \frac{1}{2}} \log (x)$; it follows that
\begin{equation}\label{eqn12:thm:holder}
\left| (t''-s)^{e_1 - \frac{1}{2}} \log (t''-s) - (t'-s)^{e_1 - \frac{1}{2}} \log (t'-s) \right| \leq  e_2^{e_1- \frac{3}{2}} \left(\big|e_1-\frac{1}{2}\big| |\log(e_2) | +1 \right) \big| t''-t'\big|\,,
\end{equation}
where 
\begin{equation}\label{eqn13:thm:holder}
e_2 \in (t'-s,t''-s) \subseteq (0,1).
\end{equation} 
Then, (\ref{eqn11:thm:holder}) and (\ref{eqn13:thm:holder}) imply that
\begin{equation}\label{eqn14:thm:holder}
 e_2^{e_1- \frac{3}{2}} \left(\big|e_1-\frac{1}{2}\big|  |\log(e_2) |+1 \right) \leq  e_2^{\underline{a}- \frac{3}{2}} \left( \log(e_2^{-1}) +1 \right) \leq (t'-s)^{\underline{a}- \frac{3}{2}} \left( \log\big((t'-s)^{-1}\big) +1 \right) \, ,
\end{equation}
where the last inequality comes from the fact that the function $x \mapsto x^{\underline{a}-\frac{3}{2}} \log(x^{-1})$ is decreasing on $(0,1)$. By combining (\ref{eqn10:thm:holder}), (\ref{eqn12:thm:holder}) and (\ref{eqn13:thm:holder}), it follows that
\begin{align}\label{eqn15:thm:holder}
& 2^{\frac{j}{2}} \int_0^{t'_j}  \Big| \big(t''-s\big)^{A(s) -\frac{1}{2}} - \big(t''-s\big)^{A(s+2^{-j-1})-\frac{1}{2}}   - \big(t'-s\big)^{A(s)-\frac{1}{2}} + \big(t'-s\big)^{A(s+2^{-j-1}) - \frac{1}{2}} \Big| ds \nonumber \\
\leq & 2^{\frac{j}{2}} \big| t''-t'\big| \int_0^{t'_j} (t'-s)^{\underline{a}- \frac{3}{2}} \left( \log\big((t'-s)^{-1}\big) +1 \right)  \big|A(s) - A(s+2^{-j-1})\big| ds  \nonumber \\
 \leq & C_1 2^{-j(\gamma - \frac{1}{2}) } \big| t''-t'\big| \int_0^{t'_j} (t'-s)^{\underline{a}- \frac{3}{2}} \left( \log\big((t'-s)^{-1}\big) +1 \right)  ds  \nonumber \\
  \leq  & C_1 2^{-j(\gamma - \frac{1}{2}) } \big| t''-t'\big| \int_0^{1} u^{\underline{a}- \frac{3}{2}} \left( \log\big(u^{-1}\big) +1 \right)  du   \nonumber \\
 = & C_2 2^{-j(\gamma - \frac{1}{2}) } \big| t''-t'\big|\,,
\end{align}
where we have used the assumption (\ref{eqn:regholdA}) on the paths of $A$, and where $C_2 = C_1 \left( (\unda-1/2)^{-1}+(\unda-1/2)^{-2} \right)$. Let us now conveniently bound from above the second integral in (\ref{eqn100:thm:holder}). There is no restriction to assume that $t'_j =t'<1-2^{-j-1}$ since this integral vanishes when $t'_j=1-2^{-j-1}$. Therefore, one has that
\begin{equation}\label{eqn16:thm:holder}
2^{\frac{j}{2}} \int_{t'_j}^{t''_j }\Big| \big(t''-s\big)^{A(s) - \frac{1}{2} } - \big(t''-s\big)^{A(s+2^{-j-1})-\frac{1}{2}}   \Big| ds = 2^{\frac{j}{2}} \int_{t'}^{t''_j}\Big| \big(t''-s\big)^{A(s) - \frac{1}{2} } - \big(t''-s\big)^{A(s+2^{-j-1})-\frac{1}{2}}   \Big| ds\,.
\end{equation}
Let us fix $s \in (t',t_j'')$ and apply the Mean Value Theorem to the function $x \mapsto (t''-s)^{x - \frac{1}{2}}$. One gets
\begin{align}\label{eqn17:thm:holder}
\Big| \big(t''-s\big)^{A(s) - \frac{1}{2} } - \big(t''-s\big)^{A(s+2^{-j-1})-\frac{1}{2}}   \Big|  = &  (t'' -s)^{e_3-\frac{1}{2}} \big| \log (t''-s) \big| \big|A(s) - A(s+2^{-j-1}) \big| \nonumber \\
\leq & (t'' -s)^{\underline{a}-\frac{1}{2}} \big| \log (t''-s) \big| \big|A(s) - A(s+2^{-j-1}) \big|\,,
\end{align}
since $e_3$ is between $A(s)$ and $A(s+2^{-j-1})$. Next, using (\ref{eqn17:thm:holder}) and the assumption (\ref{eqn:regholdA}), one obtains
\begin{align}\label{eqn18:thm:holder}
2^{\frac{j}{2}} \int_{t'}^{t_j''}\Big| \big(t''-s\big)^{A(s) - \frac{1}{2} } - \big(t''-s\big)^{A(s+2^{-j-1})-\frac{1}{2}}   \Big| ds \leq  & 2^{\frac{j}{2}} \int_{t'}^{t_j''} (t'' -s)^{\underline{a}-\frac{1}{2}} \big| \log (t''-s) \big| \big|A(s) - A(s+2^{-j-1}) \big| ds \nonumber \\
\leq & C_1 2^{-j(\gamma - \frac{1}{2}) }\int_{t'}^{t''} (t'' -s)^{\underline{a}-\frac{1}{2}} \big| \log (t''-s) \big| ds \nonumber \\
\leq & C_3 2^{-j(\gamma - \frac{1}{2}) } (t''-t' )\,,
\end{align}
where $C_3 = C_1 \sup_{x \in (0,1]}x^{\underline{a}-\frac{1}{2}} \big|\log(x)\big| < + \infty$. Next, combining (\ref{eqn100:thm:holder}), (\ref{eqn15:thm:holder}) and (\ref{eqn18:thm:holder}), it follows that 
\begin{equation}\label{eqn19:thm:holder}
 \lambda^1_{j}(t',t'') \leq C_4 2^{-j(\gamma - \frac{1}{2}) } (t''-t' )\,,
\end{equation}
where $C_4 = C_2 + C_3$. 

\bigskip

\noindent\textbf{$\bullet$ Step 2 : Upper bound for  $\lambda^2_{j}(t',t'') \, $. } \\
First, observe that one can derive from (\ref{eq:Ldef}), (\ref{eqn2:thm:holder}) and (\ref{eqn3:thm:holder}), one has
\begin{align}\label{eqn21:thm:holder}
 \lambda^2_{j}(t',t'') 
 =   2^{\frac{j}{2}} \int_{0}^{t''_j} & \Big| L_{t''}\big(s, A(s+2^{-j-1})\big) -L_{t''}\big(s+2^{-j-1}, A(s+2^{-j-1})\big) \nonumber \\
& - L_{t'}\big(s,A(s+2^{-j-1})\big)  + L_{t'}\big(s+2^{-j-1},A(s+2^{-j-1})\big) \Big|  ds \, .
\end{align}
Let us define $\nu_0$ in the following way: if $\nu_1=0$ then $\nu_0:=0$, otherwise $\nu_0$ is an arbitrary fixed strictly positive real number belonging to the open interval $(0,\nu_1)$. Then, one sets
\begin{align}\label{eqn:mu1:thm:holder}
\mu^1_j(t',t'',\nu_0) :=2^{\frac{j}{2}} \int_{0}^{\max\{0,\nu_0 - 2^{-j-1}\}} & \Big| L_{t''}\big(s, A(s+2^{-j-1})\big) -L_{t''}\big(s+2^{-j-1}, A(s+2^{-j-1})\big) \nonumber \\
 & - L_{t'}\big(s,A(s+2^{-j-1})\big)  + L_{t'}\big(s+2^{-j-1},A(s+2^{-j-1})\big) \Big|  ds\, ,
\end{align}
\begin{align}\label{eqn:mu2:thm:holder}
\mu^2_j(t',t'',\nu_0) :=2^{\frac{j}{2}} \int_{\max\{0,\nu_0 - 2^{-j-1}\} }^{\max\{0,t'- 2^{-j-1}\}} & \Big| L_{t''}\big(s, A(s+2^{-j-1})\big) -L_{t''}\big(s+2^{-j-1}, A(s+2^{-j-1})\big) \nonumber \\
 & - L_{t'}\big(s,A(s+2^{-j-1})\big)  + L_{t'}\big(s+2^{-j-1},A(s+2^{-j-1})\big) \Big|  ds\,,
\end{align}
and
\begin{align}\label{eqn:mu3:thm:holder}
\mu^3_j(t',t'') := 2^{\frac{j}{2}} \int_{\max\{0,t'-2^{-j-1}\}}^{t''_j} & \Big| L_{t''}\big(s, A(s+2^{-j-1})\big) -L_{t''}\big(s+2^{-j-1}, A(s+2^{-j-1})\big) \nonumber \\
 & - L_{t'}\big(s,A(s+2^{-j-1})\big)  + L_{t'}\big(s+2^{-j-1},A(s+2^{-j-1})\big) \Big|  ds\, .
\end{align}
Notice that one has
\begin{equation}\label{eqn:mu:thm:holder}
 \lambda^2_{j}(t',t'')  = \mu^1_j(t',t'',\nu_0) + \mu^2_j(t',t'',\nu_0) + \mu^3_j(t',t'') \, . 
\end{equation}

\medskip

\textit{1. Upper bound for $\mu^1_j(t',t'',\nu_0)\, $.} \\
Observe that $\mu^1_j(t',t'',\nu_0)=0$ when $\nu_0 - 2^{-j-1}\le 0$, so there is no restriction to assume that $j$ is such that $\nu_0 - 2^{-j-1}>0$. Then, using (\ref{eq:Ldef}), (\ref{eqn:mu1:thm:holder}) and the inequalities
\begin{equation}\label{eqn212:thm:holder}
\nu_0 < \nu_1 \leq t' \leq t'' \leq \nu_2  \, ,
\end{equation} 
one gets that
\begin{align}\label{eqn211:thm:holder}
 \mu^1_j(t',t'',\nu_0)
= 2^{\frac{j}{2}} \int_{0}^{\nu_0 - 2^{-j-1}} & \Big| (t''-s)^{A(s+2^{-j-1}) - \frac{1}{2}} -(t''-s-2^{-j-1})^{A(s+2^{-j-1})- \frac{1}{2}} \nonumber \\
& - (t'-s)^{A(s+2^{-j-1})- \frac{1}{2}}  + (t'-s-2^{-j-1})^{A(s+2^{-j-1})- \frac{1}{2}} \Big|  ds \, .
\end{align}
Let us fix $s \in (0,\nu_0 - 2^{-j-1})$ and apply the Mean Value Theorem to the function ${x \mapsto x^{A(s+ 2^{-j-1}) - \frac{1}{2}} - (x-2^{-j-1})^{A(s+ 2^{-j-1}) - \frac{1}{2}}}$. One gets that 
\begin{align}\label{eqn213:thm:holder}
& \Big| (t''-s)^{A(s+2^{-j-1}) - \frac{1}{2}} -(t''-s-2^{-j-1})^{A(s+2^{-j-1})- \frac{1}{2}}  - (t'-s)^{A(s+2^{-j-1})- \frac{1}{2}}  + (t'-s-2^{-j-1})^{A(s+2^{-j-1})- \frac{1}{2}} \Big| \nonumber \\
= &  \left| \left(A(s+2^{-j-1})- \frac{1}{2}\right)  (t''-t')\right| \left| e_4^{A(s+2^{-j-1}) - \frac{3}{2}} - (e_4 - 2^{-j-1})^{A(s+2^{-j-1}) - \frac{3}{2}}\right| \nonumber \\
 \leq & (t''-t') \left| e_4^{A(s+2^{-j-1}) - \frac{3}{2}} - (e_4 - 2^{-j-1})^{A(s+2^{-j-1}) - \frac{3}{2}}\right|\,,
\end{align}
where
\begin{equation}\label{eqn214:thm:holder}
e_4 \in (t'-s,t''-s) \, .
\end{equation}
Then, the Mean Value Theorem applied to the function $x \mapsto x^{A(s+2^{-j-1}) - \frac{3}{2}}$ implies that there exists 
\begin{equation}\label{eqn215:thm:holder}
e_5 \in (e_4-2^{-j-1},e_4)
\end{equation}
such that
\begin{align}\label{eqn216:thm:holder}
\left| e_4^{A(s+2^{-j-1}) - \frac{3}{2}} - (e_4 - 2^{-j-1})^{A(s+2^{-j-1}) - \frac{1}{2}}\right| = & 2^{-j-1} \big| A(s+2^{-j-1})- \frac{3}{2} \big| e_5^{ A(s+2^{-j-1})- \frac{5}{2}} \nonumber \\
\leq & 2^{-j-1} (\nu_1-s-2^{-j-1})^{ \underline{a}- \frac{5}{2}}\,,
\end{align}
where we have used (\ref{eqn:bornesA}), (\ref{eqn212:thm:holder}), (\ref{eqn214:thm:holder}) and (\ref{eqn215:thm:holder}). Putting together (\ref{eqn211:thm:holder}), (\ref{eqn213:thm:holder}) and (\ref{eqn216:thm:holder}), it follows that 
\begin{equation}\label{eqn217:thm:holder}
\mu^1_j(t',t'',\nu_0) \leq 2^{-\frac{j}{2}-1} (t''-t') \int_0^{\nu_0 - 2^{-j-1}}    (\nu_1-s-2^{-j-1})^{ \underline{a}- \frac{5}{2}}  ds \leq 2^{-\frac{j}{2}} (t''-t') \frac{(\nu_1-\nu_0)^{\underline{a}- \frac{3}{2}}}{3- 2\underline{a}} \, .
\end{equation}

\medskip

\textit{2. Upper bound for $\mu^2_j(t',t'',\nu_0)\, $.} \\
Observe that $\mu^2_j(t',t'',\nu_0)=0$ when $t'- 2^{-j-1}\le 0$, so there is no restriction to assume that $j$ is such that $t' - 2^{-j-1}>0$. Then, using (\ref{eq:Ldef}) and (\ref{eqn:mu2:thm:holder}), we have 
\begin{align}\label{eqn221:thm:holder}
\mu^2_j(t',t'',\nu_0)
= 2^{\frac{j}{2}} \int_{\max\{0,\nu_0 - 2^{-j-1}\}}^{t'-2^{-j-1}} & \Big| (t''-s)^{A(s+2^{-j-1}) - \frac{1}{2}} -(t''-s-2^{-j-1})^{A(s+2^{-j-1})- \frac{1}{2}} \nonumber \\
& - (t'-s)^{A(s+2^{-j-1})- \frac{1}{2}}  + (t'-s-2^{-j-1})^{A(s+2^{-j-1})- \frac{1}{2}} \Big|  ds \, .
\end{align}
Assume first that $j \in \{0,\dots, j_0-1\}$. Combining (\ref{eqn221:thm:holder}) and (\ref{thm:holder:def}) in which $\nu_1$ is replaced by $\nu_0$, we get that
\begin{align}\label{eqn225:thm:holder}
& \mu^2_j(t',t'',\nu_0) \nonumber \\
& = 2^{-\frac{j}{2}-1} \int_{\max\{0,2^{j+1}\nu_0 - 1\}}^{2^{j+1}t'-1}  \Big| (t''-2^{-j-1}u)^{A(2^{-j-1}u+2^{-j-1}) - \frac{1}{2}} -(t''-2^{-j-1}u-2^{-j-1})^{A(2^{-j-1}u+2^{-j-1})- \frac{1}{2}} \nonumber \\
& \hspace*{3cm} - (t'-2^{-j-1}u)^{A(2^{-j-1}u+2^{-j-1})- \frac{1}{2}}  + (t'-2^{-j-1}u-2^{-j-1})^{A(2^{-j-1}u+2^{-j-1})- \frac{1}{2}} \Big|  du \nonumber \\
& \leq  2^{-(j+1)A_{\nu_0,\nu_2}} \int_{\max\{0,2^{j+1}\nu_0 - 1\}}^{2^{j+1}t'-1}  \Big| (2^{j+1}t''-u)^{A(2^{-j-1}(u+1)) - \frac{1}{2}} -(2^{j+1}t''-(u+1))^{A(2^{-j-1}(u+1))- \frac{1}{2}} \nonumber \\
& \hspace*{3cm} - (2^{j+1}t'-u)^{A(2^{-j-1}(u+1))- \frac{1}{2}}  + (2^{j+1}t'-(u+1))^{A(2^{-j-1}(u+1))- \frac{1}{2}} \Big|  du \nonumber \\
& =  2^{-(j+1)A_{\nu_0,\nu_2}} \int_{0}^{\min\{2^{j+1}t'-1,2^{j+1}(t'-\nu_0)\}}  \Big| (2^{j+1}(t''-t')+v+1)^{A(t'-2^{-j-1}v) - \frac{1}{2}}\nonumber \\
& \hspace*{3cm} -(2^{j+1}(t''-t') +v)^{A(t'-2^{-j-1}v)- \frac{1}{2}} - (v+1)^{A(t'-2^{-j-1}v)- \frac{1}{2}}  + v^{A(t'-2^{-j-1}v)- \frac{1}{2}} \Big|  dv \, .
\end{align}
Next, applying on the interval $[0,2^{j+1}(t''-t')]$ the Mean Value Theorem to the function ${x \mapsto (x+v+1)^{A(t'-2^{-j-1}v) - \frac{1}{2}} - (x+v)^{A(t'-2^{-j-1}v) - \frac{1}{2}}}$, and using (\ref{eqn:bornesA}), it follows that (\ref{eqn225:thm:holder}) can be bounded from above by
\begin{equation*}\label{eqn226:thm:holder}
2^{(j+1)(1-A_{\nu_0,\nu_2})}(t''-t')  \int_{0}^{+ \infty} \Big (v^{A(t'-2^{-j-1}v)- \frac{3}{2}} - (v+1)^{A(t'-2^{-j-1}v)- \frac{3}{2}} \Big) dv\, .
\end{equation*}
Observe that, thanks to (\ref{eqn:bornesA}), this last integral can be bounded from above by a deterministic constant $c>0$ not depending on $t',t'',j$. Thus, in view of the definition (\ref{eqn224:thm:holder}) of $j_0$ and the fact that ${j\leq j_0-1}$, one gets that 
\begin{equation}\label{eqn227:thm:holder}
\mu^2_j(t',t'',\nu_0) \leq c 2^{(j+1)(1-A_{\nu_0,\nu_2})}(t''-t') \leq c(t''-t')^{A_{\nu_0,\nu_2}}  \,, \quad\mbox{for all $j \in \{0,\dots, j_0-1\}$.}
\end{equation}
Assume now that $j \geq j_0$. Let us fix $s \in (\max\{0,\nu_0 - 2^{-j-1}\},t' - 2^{-j-1})$. Applying the Mean Value Theorem to the function ${x \mapsto (t''-s-x)^{A(s+2^{-j-1}) - \frac{1}{2}} - (t'-s-x)^{A(s+2^{-j-1})- \frac{1}{2}}}$, we obtain, for some ${e_6 \in (0,2^{-j-1})}$, that
\begin{align}\label{eqn222:thm:holder}
& \left| (t''-s)^{A(s+2^{-j-1}) - \frac{1}{2}} -(t''-s-2^{-j-1})^{A(s+2^{-j-1})- \frac{1}{2}} - (t'-s)^{A(s+2^{-j-1})- \frac{1}{2}}  + (t'-s-2^{-j-1})^{A(s+2^{-j-1})- \frac{1}{2}}  \right| \nonumber \\
= & 2^{-j-1} \left|A(s+2^{-j-1})- \frac{1}{2}\right|\left| (t''-s-e_6)^{A(s+2^{-j-1}) - \frac{3}{2}} - (t'-s -e_6)^{A(s+2^{-j-1})- \frac{3}{2}}   \right| \nonumber \\
\leq & 2^{-j-1} \left( (t'-s-2^{-j-1})^{A(s+2^{-j-1}) - \frac{3}{2}} - (t''-s -2^{-j-1})^{A(s+2^{-j-1})- \frac{3}{2}}   \right)\,,
\end{align}
where we have used  (\ref{eqn:bornesA}). Consequently, (\ref{eqn221:thm:holder}) can be bounded from above by
\begin{align}\label{eqn223:thm:holder}
& 2^{-\frac{j}{2}-1} \int_{\max\{0,\nu_0 - 2^{-j-1}\}}^{t' - 2^{-j-1}}\left( (t'-s-2^{-j-1})^{A(s+2^{-j-1}) - \frac{3}{2}} - (t''-s -2^{-j-1})^{A(s+2^{-j-1})- \frac{3}{2}}   \right) ds \nonumber \\
\leq & 2^{-\frac{j}{2}-1} \int_{\nu_0}^{t' }\left( (t'-v)^{A(v) - \frac{3}{2}} - (t''-v)^{A(v)- \frac{3}{2}}   \right) dv \nonumber \\
 = & 2^{-\frac{j}{2}-1} (t''-t')^{-\frac{1}{2}} \int_{1}^{\frac{t''-\nu_0}{t''-t'}} (t''-t')^{A(t''-(t''-t')u)}\left( (u-1)^{A(t''-(t''-t')u)-\frac{3}{2}} -  u^{A(t''-(t''-t')u)-\frac{3}{2}}\right)du \nonumber \\
\le & 2^{-\frac{j}{2}-1} (t''-t')^{A_{\nu_0,\nu_2}-\frac{1}{2}}  \int_{1}^{+ \infty} \left( (u-1)^{A(t''-(t''-t')u)-\frac{3}{2}} -  u^{A(t''-(t''-t')u)-\frac{3}{2}}\right)du\,,
\end{align}
where we have used the changes of variables $v=s+2^{-j-1}$ and $u = \frac{t''-v}{t''-t'}$. Moreover, thanks to (\ref{eqn:bornesA}), the last integral can be bounded from above by a deterministic constant $c'>0$ not depending on $t',t'',j$. Then  (\ref{eqn221:thm:holder}), (\ref{eqn222:thm:holder}), and (\ref{eqn223:thm:holder}) imply
\begin{equation}\label{eqn228:thm:holder}
\mu^2_j(t',t'',\nu_0) \leq c' 2^{-\frac{j}{2}-1} (t''-t')^{A_{\nu_0,\nu_2}-\frac{1}{2}} \,, \quad\mbox{for all $j \geq j_0$.}
\end{equation}

\medskip

\textit{3. Upper bound for $\mu^3_j(t',t'')\, $.} \\
Assume first that $j \in \{0,\dots, j_0-1\}$. Let us then decompose the integral (\ref{eqn:mu3:thm:holder}) defining $\mu^3_j(t',t'') $ into two parts: the integral on the interval $[\max\{0,t'-2^{-j-1}\},\max\{0,t''-2^{-j-1}\}]$ denoted by $\mu^{3,1}_j(t',t'')$, and the integral on the interval $[\max\{0,t''-2^{-j-1}\}, t''_j]$ denoted by $\mu^{3,2}_j(t'')$. Let us now provide an appropriate upper bound for $\mu^{3,1}_j(t',t'')$.
One can assume that ${t''-2^{-j-1}\ge 0}$ since $\mu^{3,1}_j(t',t'')=0$ in the other case. Observe that one can derive from the latter inequality and from (\ref{eqn224:thm:holder}) that ${t'  \geq t''-2^{-j-1}}$. Thus, one has
\begin{align}\label{eqn232:thm:holder}
&\mu^{3,1}_j(t',t'') = 2^{\frac{j}{2}} \int_{\max\{0,t'-2^{-j-1}\}}^{t''-2^{-j-1}}  \Big| (t''-s)^{A(s+2^{-j-1})- \frac{1}{2}} -(t''-s-2^{-j-1})^{{A(s+2^{-j-1})- \frac{1}{2}}}  -(t'-s)^{A(s+2^{-j-1})- \frac{1}{2}} \Big|  ds \nonumber \\
 & \le 2^{\frac{j}{2}} \int_{t'-2^{-j-1}}^{t''-2^{-j-1}}  \Big| (t''-s)^{A(s+2^{-j-1})- \frac{1}{2}} -(t'-s)^{{A(s+2^{-j-1})- \frac{1}{2}}} \Big|ds +   2^{\frac{j}{2}} \int_{t'-2^{-j-1}}^{t''-2^{-j-1}}  (t''-s-2^{-j-1})^{A(s+2^{-j-1})- \frac{1}{2}}   ds  \, .
 \end{align}
 Using (\ref{thm:holder:def}), (\ref{eqn224:thm:holder}) and the assumption $j \leq j_0-1$, one gets that
 \begin{align}\label{eqn2321:thm:holder}
   2^{\frac{j}{2}} \int_{t'-2^{-j-1}}^{t''-2^{-j-1}}  (t''-s-2^{-j-1})^{A(s+2^{-j-1})- \frac{1}{2}}   ds
   \le &  2^{\frac{j}{2}} \int_{t'-2^{-j-1}}^{t''-2^{-j-1}}  (t''-s-2^{-j-1})^{A_{\nu_1,\nu_2}- \frac{1}{2}}   ds \nonumber \\
   = &   \frac{2^{\frac{j}{2}}}{A_{\nu_1,\nu_2}+ \frac{1}{2}} (t''-t')^{A_{\nu_1,\nu_2}+ \frac{1}{2}}  \nonumber \\
   \le & C_5 (t''-t')^{A_{\nu_1,\nu_2}}\,,
 \end{align}
where $C_5= 2^{-\frac{1}{2}} (A_{\nu_1,\nu_2}+ \frac{1}{2})^{-1}$. Using again (\ref{thm:holder:def}) and the change of variable $v=2^{j+1}(t''-s)-1$, one obtains that
\begin{align}\label{eqn233:thm:holder}
 &  2^{\frac{j}{2}} \int_{\max\{0,t'-2^{-j-1}\}}^{t''-2^{-j-1}}  \Big| (t''-s)^{A(s+2^{-j-1})- \frac{1}{2}} -(t'-s)^{{A(s+2^{-j-1})- \frac{1}{2}}} \Big|ds \nonumber \\
\le & 2^{-\frac{j}{2} -1} \int_0^{2^{j+1}(t''-t')} \Big| (2^{-j-1}(v+1))^{A(t''-2^{-j-1}v)- \frac{1}{2}} -(t'-t''+2^{-j-1}(v+1))^{{A(t''-2^{-j-1}v)- \frac{1}{2}}} \Big|dv \nonumber \\
\le &  2^{-(j+1)A_{\nu_1,\nu_2}}\int_0^{2^{j+1}(t''-t')} \Big| (v+1)^{A(t''-2^{-j-1}v)- \frac{1}{2}} -(2^{j+1}(t'-t'')+v+1)^{{A(t''-2^{-j-1}v)- \frac{1}{2}}} \Big|dv \nonumber \\
\le &  2^{-(j+1)A_{\nu_1,\nu_2}}\int_0^{2^{j+1}(t''-t')} \Big| (v+1)^{A(t''-2^{-j-1}v)- \frac{1}{2}} -(2^{j+1}(t'-t'')+v+1)^{{A(t''-2^{-j-1}v)- \frac{1}{2}}} \Big|dv \nonumber \\
\le &  2^{(j+1)(1-A_{\nu_1,\nu_2})}(t''-t') \int_0^{1}  v^{A(t''-2^{-j-1}v)- \frac{3}{2}} dv 
 \end{align}
where we have used in the last equality the Mean Value Theorem applied to the function ${x \mapsto (x+v+1)^{A(t''-2^{-j-1}v)- \frac{1}{2}}}$, (\ref{eqn224:thm:holder}) and the assumption $j \leq j_0-1$. Observe that, thanks to (\ref{eqn:bornesA}), the last integral can be bounded from above by a deterministic constant $c''>0$ not depending on $t',t'',j$. Thus, one can derive from (\ref{eqn233:thm:holder}), (\ref{eqn224:thm:holder}) and the inequality $j \leq j_0-1$, that
\begin{equation}\label{eqn2331:thm:holder}
 2^{\frac{j}{2}} \int_{\max\{0,t'-2^{-j-1}\}}^{t''-2^{-j-1}}  \Big| (t''-s)^{A(s+2^{-j-1})- \frac{1}{2}} -(t'-s)^{{A(s+2^{-j-1})- \frac{1}{2}}} \Big|ds \le c'' (t''-t')^{A_{\nu_1,\nu_2}} \, .
\end{equation}

\medskip

Let us now bound from above the integral $\mu^{3,2}_j(t'')$. One has
\begin{align}\label{eqn234:thm:holder}
\mu^{3,2}_j(t'')= 2^{\frac{j}{2}} \int_{\max\{0,t''-2^{-j-1}\}}^{t''_j} & \Big| L_{t''}\big(s, A(s+2^{-j-1})\big) -L_{t''}\big(s+2^{-j-1}, A(s+2^{-j-1})\big) \nonumber \\
 & - L_{t'}\big(s,A(s+2^{-j-1})\big)  + L_{t'}\big(s+2^{-j-1},A(s+2^{-j-1})\big) \Big|  ds \nonumber \\
 =  2^{\frac{j}{2}} \int_{\max\{0,t''-2^{-j-1}\}}^{t'_j} & \Big| (t''-s)^{A(s+2^{-j-1}) - \frac{1}{2}}  - (t'-s)^{A(s+2^{-j-1}) - \frac{1}{2}}  \Big|   +  2^{\frac{j}{2}} \int_{t'_j}^{t''_j} (t''-s)^{A(s+2^{-j-1}) - \frac{1}{2}}    ds\,. \nonumber \\
\end{align}
Observe that, one knows from (\ref{eqn:regholdA}) and (\ref{thm:holder:def}) that
\begin{equation}\label{eqn235:thm:holder}
\min\big\{A(x)\,:\, x\in [\nu_1-2^{-j-1},\nu_2+2^{-j-1}]\cap [0,1]\big\}\geq A_{\nu_1,\nu_2} - C_1 2^{-(j+1)\gamma}.
\end{equation}
Hence, one has
\begin{equation}\label{eqn236:thm:holder}
2^{-(j+1)\min\big\{A(x)\,:\, x\in [\nu_1-2^{-j-1},\nu_2+2^{-j-1}]\cap [0,1]\big\}} \leq C' 2^{-(j+1)A_{\nu_1,\nu_2}} \, ,
\end{equation}
where $C'$ is a positive almost surely finite random constant not depending on $t',t'',j$. It follows from the change of variable $u=2^{j+1}(t'-s)$, (\ref{eqn236:thm:holder}), and the Mean Value Theorem that 
\begin{align}\label{eqn237:thm:holder}
 & 2^{\frac{j}{2}} \int_{\max\{0,t''-2^{-j-1}\}}^{t'_j}  \Big| (t''-s)^{A(s+2^{-j-1}) - \frac{1}{2}}  - (t'-s)^{A(s+2^{-j-1}) - \frac{1}{2}}  \Big| ds  \nonumber \\
\leq  & 2^{-\frac{j}{2}-1} \int_{(t'-t'_j)2^{j+1}}^{1-(t''-t')2^{j+1}} \Big| (t''-t'+2^{-j-1}u)^{A(t'+(1-u)2^{-j-1}) - \frac{1}{2}}  - (2^{-j-1}u)^{A(t'+(1-u)2^{-j-1}) - \frac{1}{2}}  \Big|  du \nonumber \\
= & C' 2^{-\frac{1}{2}}  2^{-(j+1)A_{\nu_1,\nu_2}} \int_{(t'-t'_j)2^{j+1}}^{1-(t''-t')2^{j+1}} \Big| \big(2^{j+1}(t''-t')+u\big)^{A(t'+(1-u)2^{-j-1}) - \frac{1}{2}}  - u^{A(t'+(1-u)2^{-j-1}) - \frac{1}{2}}  \Big| du \nonumber \\
\le & C' 2^{-\frac{1}{2}}  2^{(j+1)(1-A_{\nu_1,\nu_2})} (t''-t') \int_{(t'-t'_j)2^{j+1}}^{1-(t''-t')2^{j+1}} u^{A(t'+(1-u)2^{-j-1}) - \frac{3}{2}}  du \nonumber \\
\le & C' 2^{-\frac{1}{2}}  2^{(j+1)(1-A_{\nu_1,\nu_2})} (t''-t') \int_0^1 u^{A(t'+(1-u)2^{-j-1}) - \frac{3}{2}}  du\,. 
\end{align}
Observe that, thanks to (\ref{eqn:bornesA}), the last integral can be bounded from above by a deterministic constant $c'''>0$ not depending on $t',t'',j$. Then, (\ref{eqn224:thm:holder}) and the assumption ${j \leq j_0-1}$ entail that
\begin{equation}\label{eqn2371:thm:holder}
 2^{\frac{j}{2}} \int_{\max\{0,t''-2^{-j-1}\}}^{t'_j}  \Big| (t''-s)^{A(s+2^{-j-1}) - \frac{1}{2}}  - (t'-s)^{A(s+2^{-j-1}) - \frac{1}{2}}  \Big| ds \le C_6 (t''-t')^{A_{\nu_1,\nu_2}} \, ,
\end{equation}
where $C_6 = C' 2^{-\frac{1}{2}} c'''$. 

Let us now provide an appropriate upper bound for the second integral in the right-hand side of (\ref{eqn234:thm:holder}). There is no restriction to assume that $t'_j=t'$ since the integral vanishes in the other case. Using the triangular inequality, the Mean Value Theorem, (\ref{eqn:bornesA}), (\ref{eqn:regholdA}), (\ref{thm:holder:def}), (\ref{eqn224:thm:holder}) and the assumption $j \leq j_0-1$, one gets
\begin{align}\label{eqn2372:thm:holder}
& 2^{\frac{j}{2}} \int_{t'}^{t''_j} (t''-s)^{A(s+2^{-j-1}) - \frac{1}{2}}    ds \nonumber \\
 \le & 2^{\frac{j}{2}} \int_{t'}^{t''_j} \left| (t''-s)^{A(s+2^{-j-1}) - \frac{1}{2}}    - (t''-s)^{A(s) - \frac{1}{2}}  \right| ds + 2^{\frac{j}{2}} \int_{t'}^{t''_j} (t''-s)^{A(s) - \frac{1}{2}}    ds \nonumber \\
 \le & 2^{\frac{j}{2}} \int_{t'}^{t''_j} \left|A(s+2^{-j-1})-A(s) \right| (t''-s)^{\underline{a} - \frac{1}{2}}   \left|\log(t''-s) \right|    ds + 2^{\frac{j}{2}} \int_{t'}^{t''} (t''-s)^{A_{\nu_1,\nu_2} - \frac{1}{2}}    ds \nonumber \\
  \le &  C_3 2^{-j(\gamma-\frac{1}{2})} (t''-t')+ 2^{\frac{j}{2}}  (t''-t')^{A_{\nu_1,\nu_2} + \frac{1}{2}}  \nonumber \\
  \le & C_3 2^{-j(\gamma-\frac{1}{2})} (t''-t')+ 2^{-\frac{1}{2}}  (t''-t')^{A_{\nu_1,\nu_2} }
\end{align}
where $C_3$ is the same random constant as in (\ref{eqn18:thm:holder}).
Putting together, (\ref{eqn2321:thm:holder}), (\ref{eqn2331:thm:holder}), (\ref{eqn2371:thm:holder}) and (\ref{eqn2372:thm:holder}), one gets that
\begin{equation}\label{eqn2373:thm:holder}
\mu_j^3(t',t'') \leq C_7 (t''-t')^{A_{\nu_1,\nu_2}} + C_3 2^{-j(\gamma-\frac{1}{2})} (t''-t') \, ,\quad\mbox{for all $j \in \{0,\dots, j_0-1\}$,}
\end{equation}
where the finite random constants $C_3$ and $C_7= C_5 + c'' + C_6 + 2^{-\frac{1}{2}}$ do not depend on $t',t'',j,j_0$. 

\bigskip

It remains for us to provide an appropriate upper bound for $\mu^3_j(t',t'')$ 
in the case where $j \geq j_0$. One knows from the latter inequality and from (\ref{eqn224:thm:holder})  that ${t'-2^{-j-1} < t' < t''-2^{-j-1} \leq t''_j}$. Thus, it results from (\ref{eqn:mu3:thm:holder}), (\ref{eq:Ldef}) and the triangular inequality that
\begin{align}\label{eqn242:thm:holder}
\mu^3_j(t',t'') = & 2^{\frac{j}{2}} \int_{\max\{0,t'-2^{-j-1}\}}^{t''_j}  \Big| (t''-s)^{A(s+2^{-j-1})- \frac{1}{2}}  -(t''-s-2^{-j-1})_+^{A(s+2^{-j-1})- \frac{1}{2}}  - (t'-s)_+^{A(s+2^{-j-1})- \frac{1}{2}} \Big|  ds \nonumber \\
 \le & 2^{\frac{j}{2}} \int_{t'-2^{-j-1}}^{t''_j}  \Big| (t''-s)^{A(s+2^{-j-1})- \frac{1}{2}}  -(t''-s-2^{-j-1})_+^{A(s+2^{-j-1})- \frac{1}{2}}\Big| ds  \nonumber \\
 & \qquad  +   2^{\frac{j}{2}} \int_{t'-2^{-j-1}}^{t'}  (t'-s)^{A(s+2^{-j-1})- \frac{1}{2}} ds \nonumber \\
\le & 2^{\frac{j}{2}} \int_{t'-2^{-j-1}}^{t''-2^{-j-1}}  \Big| (t''-s)^{A(s+2^{-j-1})- \frac{1}{2}}  -(t''-s-2^{-j-1})^{A(s+2^{-j-1})- \frac{1}{2}}\Big| ds \nonumber \\
& \qquad  +  2^{\frac{j}{2}} \int_{t''-2^{-j-1}}^{t''_j}   (t''-s)^{A(s+2^{-j-1})- \frac{1}{2}} ds  +  2^{\frac{j}{2}} \int_{t'-2^{-j-1}}^{t'}  (t'-s)^{A(s+2^{-j-1})- \frac{1}{2}}   ds\, .
\end{align}
Let now bound from above in a suitable way each of the three integrals appearing in (\ref{eqn242:thm:holder}). Using standard computations, (\ref{eqn236:thm:holder}) and (\ref{thm:holder:def}), one gets that
\begin{align}\label{eqn243:thm:holder}
2^{\frac{j}{2}} \int_{t''-2^{-j-1}}^{t''_j}   (t''-s)^{A(s+2^{-j-1})- \frac{1}{2}} ds & \leq  2^{\frac{j}{2}} \int_{t''-2^{-j-1}}^{t''_j}  2^{-(j+1)(A(s+2^{-j-1})- \frac{1}{2})} ds \nonumber \\
& \leq C' 2^{j+\frac{1}{2}} (t''_j-t''+2^{-j-1})  2^{-(j+1)A_{\nu_1,\nu_2}} \nonumber \\
& \leq C' 2^{-\frac{1}{2}} 2^{-(j+1)A_{\nu_1,\nu_2}}
\end{align}
and
\begin{align}\label{eqn244:thm:holder}
2^{\frac{j}{2}} \int_{t'-2^{-j-1}}^{t'}  (t'-s)^{A(s+2^{-j-1})- \frac{1}{2}}   ds  & \leq 2^{\frac{j}{2}} \int_{t'-2^{-j-1}}^{t'}  (t'-s)^{A_{\nu_1,\nu_2}-\frac{1}{2}}   ds \nonumber \\
& \leq 2^{\frac{j}{2}} \int_{t'-2^{-j-1}}^{t'}  2^{-(j+1)(A_{\nu_1,\nu_2}-\frac{1}{2})}   ds \nonumber \\
& =  2^{-\frac{1}{2}}  2^{-(j+1)A_{\nu_1,\nu_2}} \,.
\end{align}
Moreover, the Mean Value Theorem, standard computations, and (\ref{thm:holder:def}) allow us to obtain that
\begin{align}\label{eqn245:thm:holder}
& 2^{\frac{j}{2}} \int_{t'-2^{-j-1}}^{t''-2^{-j-1}}  \Big| (t''-s)^{A(s+2^{-j-1})- \frac{1}{2}}  -(t''-s-2^{-j-1})^{A(s+2^{-j-1})- \frac{1}{2}}\Big| ds \nonumber \\
\le & 2^{-\frac{j}{2}-1} \int_{t'-2^{-j-1}}^{t''-2^{-j-1}}  (t''-s-2^{-j-1})^{A(s+2^{-j-1})- \frac{3}{2}}   ds \nonumber \\
= & 2^{-\frac{j}{2}-1} \int_{t'}^{t''}  (t''-s)^{A(s)- \frac{3}{2}}   ds \nonumber \\
\le & 2^{-\frac{j}{2}-1} \int_{t'}^{t''}  (t''-s)^{A_{\nu_1,\nu_2} - \frac{3}{2}}   ds \nonumber \\
= & \frac{2^{-\frac{j}{2}-1}}{A_{\nu_1,\nu_2} - \frac{1}{2}} (t''-t')^{A_{\nu_1,\nu_2} - \frac{1}{2}} \, .
\end{align}
Combining (\ref{eqn242:thm:holder}), (\ref{eqn243:thm:holder}), (\ref{eqn244:thm:holder}) and (\ref{eqn245:thm:holder}), it follows that
\begin{align}\label{eqn246:thm:holder}
\mu^3_j(t',t'')
\le & (C_5 +1)  2^{-\frac{1}{2}} 2^{-(j+1)A_{\nu_1,\nu_2}} + \frac{2^{-\frac{j}{2}-1}}{A_{\nu_1,\nu_2} - \frac{1}{2}} (t''-t')^{A_{\nu_1,\nu_2} - \frac{1}{2}} \,,  \quad\mbox{for all $j \geq j_0$.}
\end{align}

\medskip

\textit{4. Conclusion of the second step.} 

In the case where $j \in \{0, \dots, j_0-1\}$, putting together (\ref{eqn:mu:thm:holder}), (\ref{eqn217:thm:holder}), (\ref{eqn227:thm:holder}) and (\ref{eqn2373:thm:holder}), one obtains 
\begin{equation}\label{eqn251:thm:holder}
\lambda^2_j(t',t'') \le C_8 \left( 2^{-j(\gamma-\frac{1}{2})} (t''-t') + (t''-t')^{A_{\nu_0,\nu_2}}\right) \, ,
\end{equation}
where the finite random constant ${C_8= \frac{(\nu_1-\nu_0)^{\underline{a}- \frac{3}{2}}}{3-2\underline{a}} +c+C_7+C_3}$ does not depend on $t',t'',j,j_0$.

In the other case where $j \ge j_0$, combining (\ref{eqn:mu:thm:holder}), (\ref{eqn217:thm:holder}), (\ref{eqn228:thm:holder}) and  (\ref{eqn246:thm:holder}), one gets that
\begin{equation}\label{eqn252:thm:holder}
\lambda^2_j(t',t'') \le C_9 \left( 2^{-\frac{j}{2}} (t''-t')+   2^{-\frac{j}{2}+1}(t''-t')^{A_{\nu_0,\nu_2} - \frac{1}{2}}+  2^{-(j+1)A_{\nu_1,\nu_2}}  \right) \, ,
\end{equation}
where the finite random constant $C_9= \frac{(\nu_1-\nu_0)^{\underline{a}- \frac{3}{2}}}{3-2\underline{a}} +\frac{c'}{2}+ (C_5+1)2^{-\frac{1}{2}} + \frac{1}{2A_{\nu_1,\nu_2} -1}$ does not depend on $t',t'',j,j_0$.

\bigskip

\noindent\textbf{$\bullet$ Step 3 : Conclusion. } \\
Putting together (\ref{eqn0:thm:holder}), (\ref{eqn:thm:holder}), (\ref{eqn19:thm:holder}), (\ref{eqn251:thm:holder}) and (\ref{eqn252:thm:holder}), one obtains that
\begin{align}\label{eqn30:thm:holder}
& \sum_{j=0}^{+ \infty} \sum_{k=0}^{2^j-1} \left| \big\langle K_{t''}, h_{j,k}\big\rangle -\big\langle K_{t'}, h_{j,k}\big\rangle \right| | \varepsilon_{j,k}|
\nonumber\\
  \leq &C_\ast \sum_{j=0}^{j_0-1} \sqrt{j+1}\left(  (C_4 +C_8) 2^{-j(\gamma - \frac{1}{2}) } (t''-t' ) + C_8  (t''-t')^{A_{\nu_0,\nu_2}}\right)  \nonumber \\
& + C_\ast \sum_{j=j_0}^{+ \infty} \sqrt{j+1}\left( C_4 2^{-j(\gamma - \frac{1}{2}) } (t''-t' ) + C_9 \left( 2^{-\frac{j}{2}} (t''-t')+   2^{-\frac{j}{2}+1}(t''-t')^{A_{\nu_0,\nu_2} - \frac{1}{2}}+  2^{-(j+1)A_{\nu_1,\nu_2}}  \right)\right) \, .
\end{align}
Moreover, setting $C_{10}= C_4+2  C_8$ and using  (\ref{eqn224:thm:holder}), it follows that
\begin{align}\label{eqn31:thm:holder}
& \sum_{j=0}^{j_0-1} \sqrt{j+1}\left(  (C_4 +C_8) 2^{-j(\gamma - \frac{1}{2}) } (t''-t' ) + C_8  (t''-t')^{A_{\nu_0,\nu_2}}\right)  \nonumber \\
\leq & \sqrt{j_0}(t''-t')^{A_{\nu_0,\nu_2}}\sum_{j=0}^{j_0-1} \left(  (C_4+C_8) 2^{-j(\gamma - \frac{1}{2}) } (t''-t' )^{1- A_{\nu_0,\nu_2}} + C_8 \right)  \nonumber \\
\leq & \sqrt{\left|\log_2(t''-t')\right|}(t''-t')^{A_{\nu_0,\nu_2}} C_{10}\, j_0\nonumber \\
\leq & C_{10}\left|\log_2(t''-t')\right|^{\frac{3}{2}}(t''-t')^{A_{\nu_0,\nu_2}}
\end{align}
and
\begin{align}\label{eqn32:thm:holder}
 & \sum_{j=j_0}^{+ \infty} \sqrt{j+1}\left( C_4 2^{-j(\gamma - \frac{1}{2}) } (t''-t' ) + C_9 \left( 2^{-\frac{j}{2}} (t''-t')+   2^{-\frac{j}{2}+1}(t''-t')^{A_{\nu_0,\nu_2} - \frac{1}{2}}+  2^{-(j+1)A_{\nu_1,\nu_2}}  \right)\right) \nonumber \\
\le & (C_4+C_9) \bigg (\sum_{j=0}^{+ \infty} 2^{-j(\gamma - \frac{1}{2}) }\,\sqrt{j+1}\bigg)   (t''-t' ) + 2^{\varepsilon+ \frac{1}{2}} C_9\bigg (\sum_{j=0}^{+ \infty} 2^{-\varepsilon j}\,\sqrt{j+1} \bigg)   (t''-t')^{(A_{\nu_0,\nu_2}- \varepsilon)} \nonumber \\
& \qquad +  C_9 \bigg(\sum_{j=0}^{+ \infty}  2^{-\varepsilon j}\,\sqrt{j+1}\bigg)(t''-t')^{(A_{\nu_1,\nu_2} - \varepsilon)}\,, 
\end{align}
where $\varepsilon$ is an arbitrarily small fixed positive real number. 

The inequalities (\ref{eqn30:thm:holder}), (\ref{eqn31:thm:holder}) and (\ref{eqn32:thm:holder}) show that the stochastic process $\{X(t) : t\in I\}$ satisfies almost surely a uniform H\"older condition of order $(A_{\nu_0,\nu_2} -  \varepsilon)$ on the interval $[\nu_1,\nu_2]$. Therefore, one has almost surely that $\beta_X([\nu_1,\nu_2]) \geq (A_{\nu_0,\nu_2} -  \varepsilon)$. This implies that (\ref{eq:A:thm:holder}) is satisfied, since $(A_{\nu_0,\nu_2} -  \varepsilon)$ goes to $A_{\nu_1,\nu_2}$ when $\nu_0$ approaches $\nu_1$ and $\varepsilon$ approaches $0$.
\end{proof}

\section{Simulations}
\label{sec:Sim}

In order to simulate paths of the MPRE $\{X(t):t\in I\}$ defined in~(\ref{def_X}), it seems natural to use its approximation $\{\widetilde{X}^J(t):t\in I\}$ given by (\ref{eqn:XJt}). But so far, we only know that, for each fixed $t\in I$, the random variable $\widetilde{X}^J(t)$ converges in $L^2 (\Omega)$ to the random variable $X(t)$, when $J$ goes to $+\infty$ (see Lemma~\ref{lem:convXJt}). The first main goal of the present section is to show that this weak convergence result can be greatly improved. In fact, Proposition~\ref{prop:conv:XtildeX}, stated below, shows that the convergence also holds almost surely and uniformly in $t\in I$. It is worth mentioning that the main ingredient of the proof of this proposition is Theorem \ref{thm:Haarrepres} which has been obtained in Section~\ref{sec:repX} thanks to the Haar basis. Also, we mention that another ingredient of the proof of Proposition~\ref{prop:conv:XtildeX} is the following classical theorem. 

\begin{thm}\label{thm:levymod}{(Lévy modulus of continuity for Brownian Motion) \cite[Theorem 9.25]{karatzas2012brownian}}
 Let $B$ be the Brownian Motion in~(\ref{def_X}). There exists an event  $\widehat{\Omega} \subseteq \Omega$ of probability 1 such that, for any $\omega \in \widehat{\Omega}$, one has
 \begin{equation}
 \limsup_{\delta\rightarrow 0^+}\Bigg\{\frac{\max\big\{|B(s'', \omega)-B(s',\omega)|:\, (s'',s')\in I^2\,\mbox{ and }\,|s''-s'|\le\delta\big\}}{\sqrt{2\delta \log(1/\delta)}}\Bigg\}=1.
 \end{equation}
\end{thm}

\begin{proposition}\label{prop:conv:XtildeX}
Let  $\Omega_2$ be the event of probability 1 defined as $\Omega_2=\Omega_{**} \cap \widehat{\Omega}$, where $\Omega_{**}$ is the same event of probability~1 as in Theorem~\ref{thm:Haarrepres}. Then, for any $\omega \in \Omega_2$, one has
\begin{equation}\label{eq:conv:XtildeX}
\lim_{J\rightarrow \ii} \left\{\sup_{t\in I} \big|\widetilde{X}^J(t, \omega)-X(t,\omega)\big| \right\}=0 \, .
\end{equation}
\end{proposition}

\begin{proof}[Proof of Proposition~\ref{prop:conv:XtildeX}]
In view of Theorem \ref{thm:Haarrepres} and (\ref{eqn:XJt2}), it is sufficient to show that,  on the  event $\Omega_2$, one has
\begin{equation}\label{eqn2:prop:conv:XtildeX}
\lim_{J\rightarrow \ii} \left\{ \sup_{t\in I} \big|\widetilde{X}^J(t)-X^J(t)\big|\right\} = 0.
\end{equation}
Let $J \in \Z_+$ be arbitrary and fixed. Using the triangular inequality, (\ref{eqn:XJt3}) and Theorem \ref{thm:levymod}, one gets
\begin{align}\label{eqn:prop:conv:XtildeX}
 \big|\widetilde{X}^J(t)-X^J(t)\big| &  = \left|\sum_{l=0}^{[2^Jt]} \left(K_t(\delta_{J,l})-\overline{K}_t^{J,l} \right) \Delta B_{J,l} \right| \le \sum_{l=0}^{[2^Jt]} \left|K_t(\delta_{J,l})-\overline{K}_t^{J,l} \right| \left|\Delta B_{J,l}\right| \nonumber \\
& \le \hat{C} 2^{-\frac{J}{2}} \sqrt{1+J} \sum_{l=0}^{[2^Jt]} \left|K_t(\delta_{J,l})-\overline{K}_t^{J,l} \right| \nonumber \\
& = \hat{C} 2^{-\frac{J}{2}}  \sqrt{1+J} \left( \left|K_t(\delta_{J,[2^Jt]})-\overline{K}_t^{J,[2^Jt]} \right| + \sum_{l=0}^{[2^Jt]-1} \left|K_t(\delta_{J,l})-\overline{K}_t^{J,l} \right| \right) \, ,
\end{align}
where 
$\displaystyle \hat{C}:=\sup_{j\in \Z_+, l\in \{0,\ldots 2^j-1\}} |\Delta B_{j,l}|2^{\frac{j}{2}} (1+j)^{- \frac{1}{2}}<\infty$. 
One knows from the inequality (\ref{maj32:lem:conv:XJtXt}) in the Appendix~\ref{sec:Appen} that $\big|K_t(\delta_{J,[2^Jt]})-\overline{K}_t^{J,[2^Jt]} \big|\le 2$. Let us provide a suitable upper bound for
$$
\sum_{l=0}^{[2^Jt]-1} \left|K_t(\delta_{J,l})-\overline{K}_t^{J,l} \right|.
$$
0bserve that, for any $l\in\big\{0,\ldots,  [2^Jt]-1\big\}$, one has $t\ge\delta_{J,l+1}$. Thus, it results from Lemma~\ref{lem_accroissK},  (\ref{eqn:mvKt}) and Condition~(\ref{eqn:regholdA})  that  
\begin{align*}
 \sum_{l=0}^{[2^Jt]-1}  \left|K_t(\delta_{J,l})-\overline{K}_t^{J,l} \right| & \le c_0 2^J \sum_{l=0}^{[2^Jt]-1}  \left( \int_{\delta_{J,l}}^{\delta_{J,l+1}} |A(s)-A(\delta_{J,l})| ds + \int_{\delta_{J,l}}^{\delta_{J,l+1}} (t-s)^{\unda-\frac{3}{2}}(s-\delta_{J,l}) ds  \right) \nonumber \\
& \le C_0^* 2^J \sum_{l=0}^{[2^Jt]-1}  \int_{\delta_{J,l}}^{\delta_{J,l+1}} (s- \delta_{J,l} )^{\gamma} ds + c_0 \sum_{l=0}^{[2^Jt]-1}  \int_{\delta_{J,l}}^{\delta_{J,l+1}} (t-s)^{\unda-\frac{3}{2}}  ds \nonumber \\
& = \frac{C_0^*}{\gamma+1} 2^J 2^{-J(1+\gamma)} [2^Jt] +  c_0 \int_{0}^{\delta_{J,[2^Jt]}} (t-s)^{\unda-\frac{3}{2}}  ds \nonumber \\
& \le \frac{C_0^*}{\gamma+1} 2^{J(1-\gamma)} + \frac{c_0}{\unda-\frac{1}{2}} t^{\unda-\frac{1}{2}} \le  \frac{C_0^*}{\gamma+1} 2^{J(1-\gamma)} + \frac{c_0}{\unda-\frac{1}{2}} \, ,
 \end{align*}
where $C_0^*:=c_0 C_1$. Therefore, one deduces from (\ref{eqn:prop:conv:XtildeX}) and the last inequality that
\begin{align*}
 |\widetilde{X}^J(t)-X^J(t)|  & \le \hat{C} 2^{-\frac{J}{2}} \sqrt{1+J} \left( \left|K_t(\delta_{J,[2^Jt]})-\overline{K}_t^{J,[2^Jt]} \right| + \sum_{l=0}^{[2^Jt]-1} \left|K_t(\delta_{J,l})-\overline{K}_t^{J,l} \right| \right) \nonumber \\
& \le \hat{C} 2^{-\frac{J}{2}} \sqrt{1+J} \left( 2 + \frac{c_0^*}{\gamma+1} 2^{J(1-\gamma)} + \frac{c_0}{\unda-\frac{1}{2}}  \right) \, .
\end{align*}
Then using the inequality $\gamma>\frac{1}{2}$, one obtains (\ref{eqn2:prop:conv:XtildeX}).
\end{proof}

The second main goal of the present section is to show that simulating paths of the MPRE $\{X(t):t\in I\}$, defined in~(\ref{def_X}), can also be done by approximating it by the stochastic process $\{\widehat{X}^J(t):t\in I \}$ defined as follows: 

\begin{definition}
For each fixed $J\in\Z_+$, the stochastic process $\{\widehat{X}^J(t):t\in I \}$ is defined, for all $t\in I$, by:
\begin{equation*}
\widehat{X}^J(t):=\sum_{l=0}^{2^J-1} \widehat{K}_t^{J,l} \Delta B_{J,l},
\end{equation*}
where, for every $l\in\{0,\ldots ,2^J-1\}$, $\Delta B_{J,l}$ is as in (\ref{def:delta}) and 
\begin{align}
\label{eq:defKchap}
\widehat{K}_t^{J,l}& :=2^J\int_{\delta_{J,l}}^{\delta_{J,l+1}} (t-s)_+^{A(\delta_{J,l})-\frac{1}{2}} ds\\
& = \frac{2^J}{A(\delta_{J,l})+\frac{1}{2}} \left( (t-\delta_{J,l})_+^{A(\delta_{J,l})+\frac{1}{2}} - (t-\delta_{J,l+1})_+^{A(\delta_{J,l})+\frac{1}{2}} \right).
\end{align}
\end{definition}
The following proposition, whose proof mainly relies on Proposition~\ref{prop:conv:XtildeX}, shows that, when $J$ goes to $+\infty$, $\{\widehat{X}^J(t):t\in I \}$ converges to $\{X(t):t\in I\}$ almost surely and uniformly in $t\in I$.

\begin{proposition}
Let  $\Omega_2$ be the same event of probability~1 as in Proposition~\ref{prop:conv:XtildeX}. Then, for any   $\omega \in \Omega_2$, one has
\begin{equation}
\lim_{J\rightarrow \ii} \left\{ \sup_{t\in I } \big|\widehat{X}^J(t,\omega)-X(t,\omega)\big| \right\}=0.
\end{equation}
\end{proposition}

\begin{proof}
In view of Proposition \ref{prop:conv:XtildeX}, it enough to prove that, on $\Omega_2$, one has
\begin{equation}\label{Ant:eq:diffhattilde}
\lim_{J\rightarrow \ii} \left\{ \sup_{t\in I} \big|\widehat{X}^J(t)-\widetilde{X}^J(t)\big| \right\} = 0.
\end{equation}
Let $J \in \Z_+$ be arbitrary and fixed. Similarly to (\ref{eqn:prop:conv:XtildeX}), it can be shown that the following inequality holds on $\Omega_2$:
\begin{align}
\label{A:eq:bound3}
 \big|\widehat{X}^J(t)-X^J(t)\big| & 
\leq \hat{C} 2^{-\frac{J}{2}}\sqrt{1+J}  \left( \left|K_t(\delta_{J,[2^Jt]})-\widehat{K}_t^{J,[2^Jt]} \right| + \sum_{l=0}^{[2^Jt]-1} \left|K_t(\delta_{J,l})-\widehat{K}_t^{J,l} \right| \right) \, ,
\end{align}
where $\hat{C}$ is the same as in (\ref{eqn:prop:conv:XtildeX}). One can derive from Remark~\ref{rem:L} and the triangular inequality that 
\begin{equation}
\label{A:eq:bound2}
\left|K_t(\delta_{J,[2^Jt]})-\widehat{K}_t^{J,[2^Jt]} \right| \leq 2^J \int_{\delta_{J,[2^Jt]}}^{\delta_{J, [2^Jt]+1}} \Big| L_t\big(\delta_{J,[2^Jt]}, A(\delta_{J,[2^Jt]})-\frac{1}{2}\big) - L_t\big(s, A(\delta_{J,[2^Jt]})-\frac{1}{2}\big)\Big| ds 
\le 2\,.
\end{equation}
Let us provide a suitable upper bound for
$$
\sum_{l=0}^{[2^Jt]-1} \left|K_t(\delta_{J,l})-\widehat{K}_t^{J,l} \right|\, .
$$
Recall that for any $l\in\big\{0,\ldots,  [2^Jt]-1\big\}$, one has $t\ge\delta_{J,l+1}$. Thus, it results from (\ref{def_Kt}), (\ref{eq:defKchap}) and the Mean Value Theorem that
\begin{align*}
 \sum_{l=0}^{[2^Jt]-1}  \left|K_t(\delta_{J,l})-\widehat{K}_t^{J,l} \right| & \le 2^J \sum_{l=0}^{[2^Jt]-1} \int_{\delta_{J,l}}^{\delta_{J,l+1}} |(t-\delta_{J,l})^{A(\delta_{J,l})-\frac{1}{2}}-(t-s)^{A(\delta_{J,l})-\frac{1}{2}}| ds \nonumber \\
& \le 2^J \sum_{l=0}^{[2^Jt]-1} \int_{\delta_{J,l}}^{\delta_{J,l+1}} (s-\delta_{J,l}) (t-s)^{A(\delta_{J,l})-\frac{3}{2}} ds \nonumber \\
& \le \sum_{l=0}^{[2^Jt]-1} \int_{\delta_{J,l}}^{\delta_{J,l+1}} (t-s)^{A(\delta_{J,l})-\frac{3}{2}} ds \nonumber \\
& \le \int_{0}^{\delta_{J,[2^Jt]}} (t-s)^{\underline{a}-\frac{3}{2}} ds  = \frac{1}{\underline{a}-\frac{1}{2}} \left( t^{\underline{a}-\frac{1}{2}} - (t-\delta_{J,[2^Jt]})^{\underline{a}-\frac{1}{2}} \right)  \nonumber \\
& \le \frac{1}{\unda-\frac{1}{2}} \, .
 \end{align*}
Finally, combining the last inequality with (\ref{A:eq:bound2}) and (\ref{A:eq:bound3}), one gets that
\begin{align*}
 |\widehat{X}^J(t)-X^J(t)| &  \le \hat{C} 2^{-\frac{J}{2}} \sqrt{1+J} \left( \left|K_t(\delta_{J,[2^Jt]})-\overline{K}_t^{J,[2^Jt]} \right| + \sum_{l=0}^{[2^Jt]-1} \left|K_t(\delta_{J,l})-\overline{K}_t^{J,l} \right| \right) \nonumber \\
& \le \hat{C} 2^{-\frac{J}{2}}\sqrt{1+J}  \left( 2 + \frac{1}{\unda-\frac{1}{2}} \right) \, ,
\end{align*}
which shows that the (\ref{Ant:eq:diffhattilde}) holds.
\end{proof}

Before concluding this section let us provide some simulations of the processes $\{\widetilde{X}^J(t) : t\in I \}$ and $\{\widehat{X}^J(t) : t\in I \}$ with a random parameter $\{A(s):s\in I\}$ chosen so that, for all $s\in I$, one has: 
$$
A(s)=\overline{\underline{ R_H(s) }}_{\, a}^{\, b}:=a+(b-a)\left(\frac{R_H(s)-\min_{x\in I} R_H(x)}{\max_{x\in I} R_H(x) - \min_{x\in I} R_H(x)}\right)\,,
$$
where $a,b$ are two arbitrary fixed real numbers satisfying $1/2<a\le b<1$, and where $R_H$ is the Riemann-Liouville process of Hurst parameter $H$ introduced in (\ref{def_RiemLiou}).
In these simulations two significantly different values for Hurst parameter $H$ of $R_H$ have be taken, namely $H=0.58$ and $H=0.9$. These simulations tend to confirm a phenomenon which somehow has already appeared in the statement of Theorem~\ref{thm:holder}, namely that there are close connections between path behavior of the MPRE $\{X(t) : t\in I \}$ and the random values taken by its parameter $\{A(s):s\in I\}$: the paths of $\{X(t) : t\in I \}$ are rather rough (resp. rather smooth) in the neighbourhood of the points where the values of $\{A(s):s\in I\}$ are close to $1/2$ (resp. close to $1$). 


\begin{figure}[!ht]
  \begin{center}
    \subfloat[$A(\cdot)$ with $H=0.58$]{
      \includegraphics[width=0.47\textwidth]{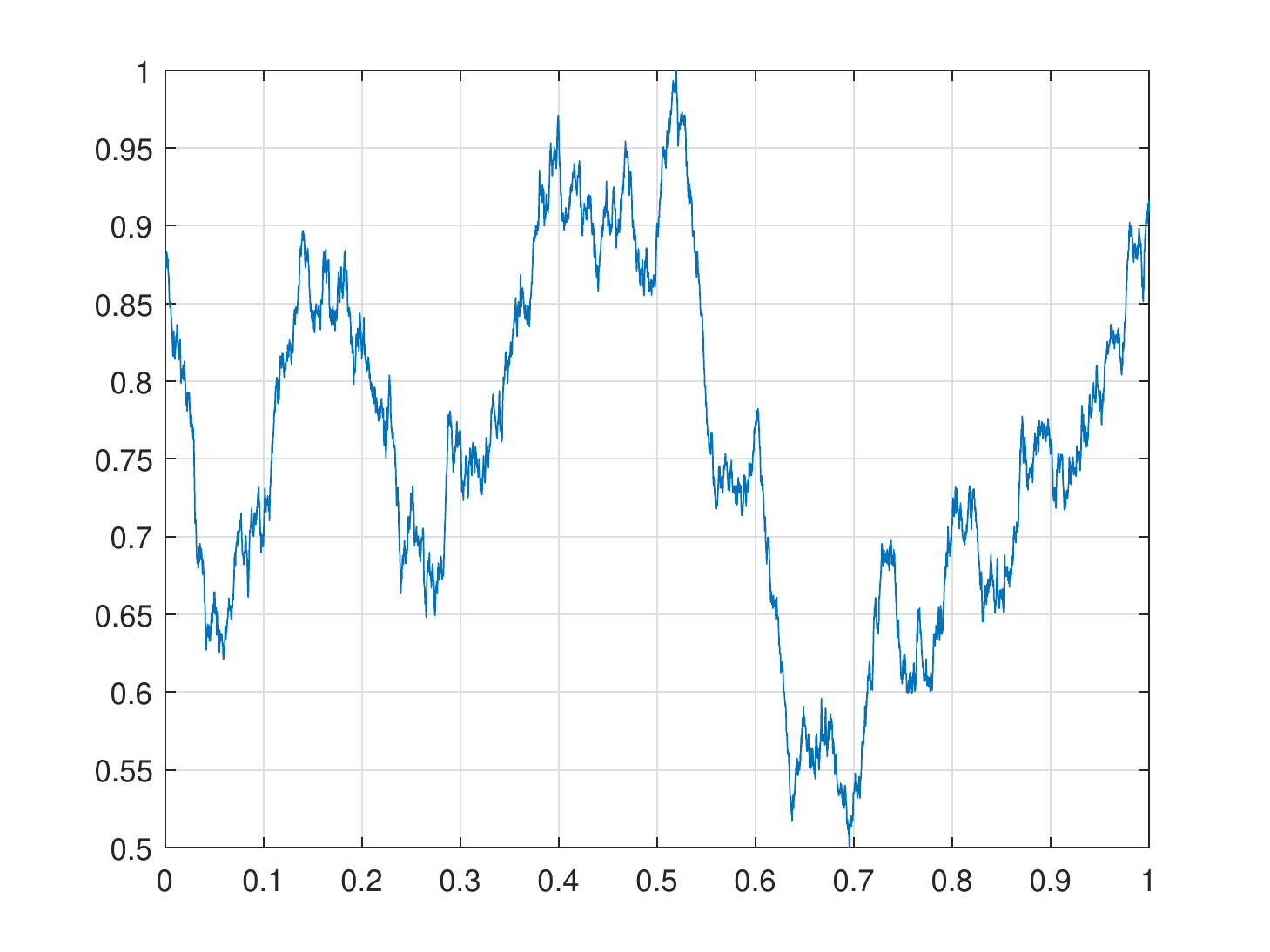}
                         }
    \subfloat[Superposition of realizations of $\widehat{X}^J$ and $\widetilde{X}^J$]{
      \includegraphics[width=0.47\textwidth]{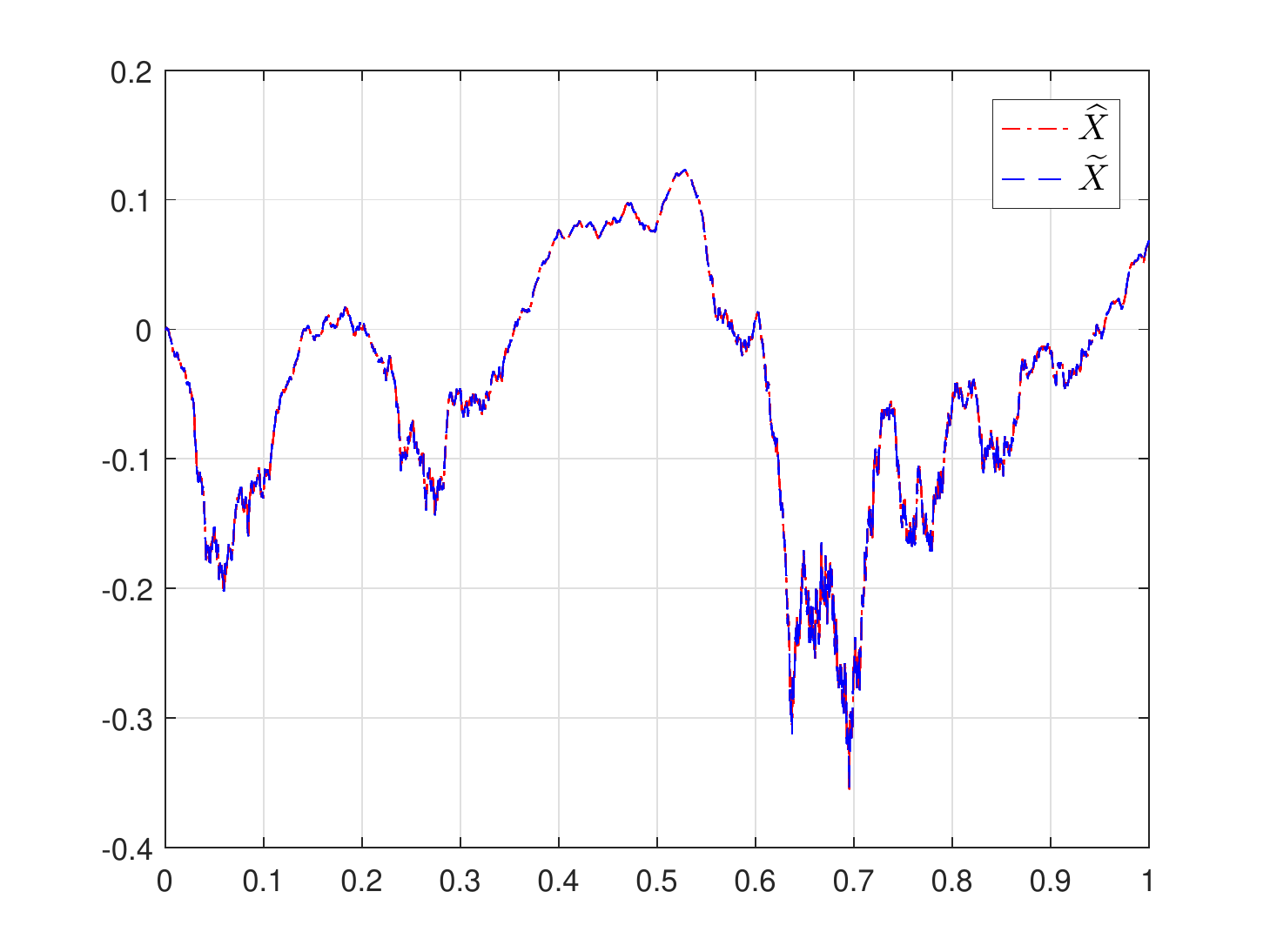}
                         }\\
  
    \subfloat[$\widehat{X}^J$]{
      \includegraphics[width=0.47\textwidth]{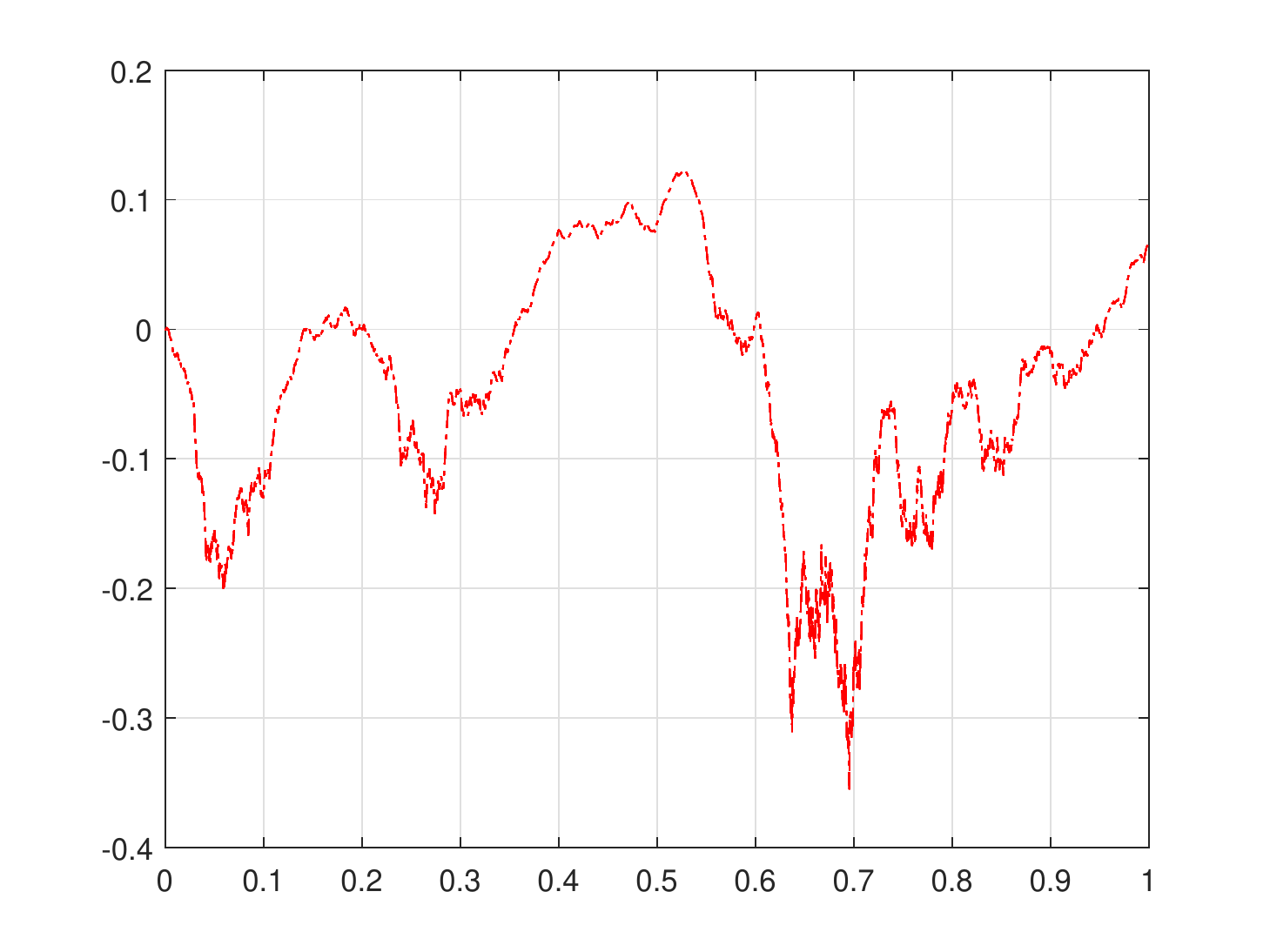}
                         }
    \subfloat[$\widetilde{X}^J$]{
      \includegraphics[width=0.47\textwidth]{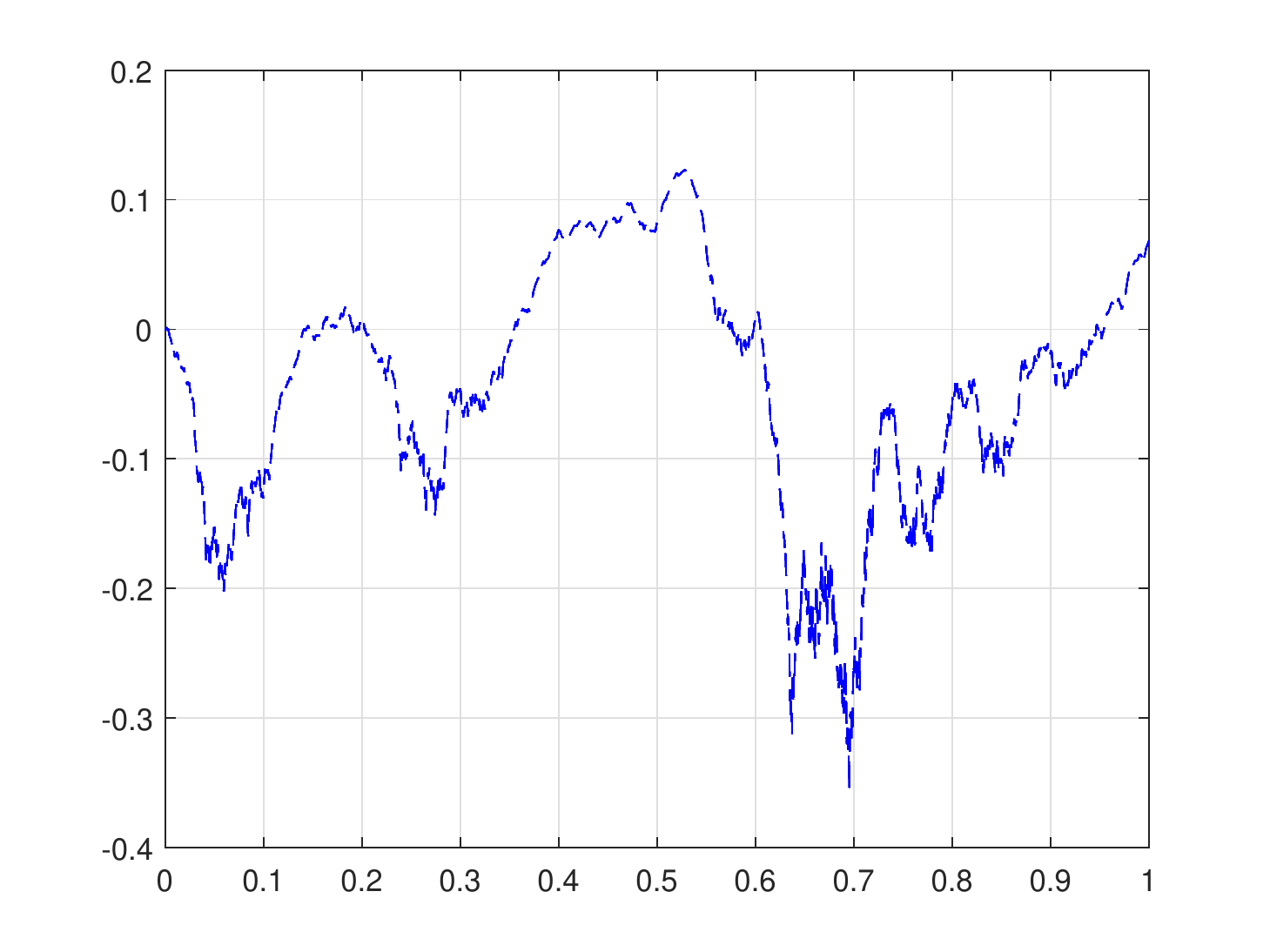}
                         }                       
                         

    \subfloat[$A(\cdot)$ with $H=0.9$]{
      \includegraphics[width=0.47\textwidth]{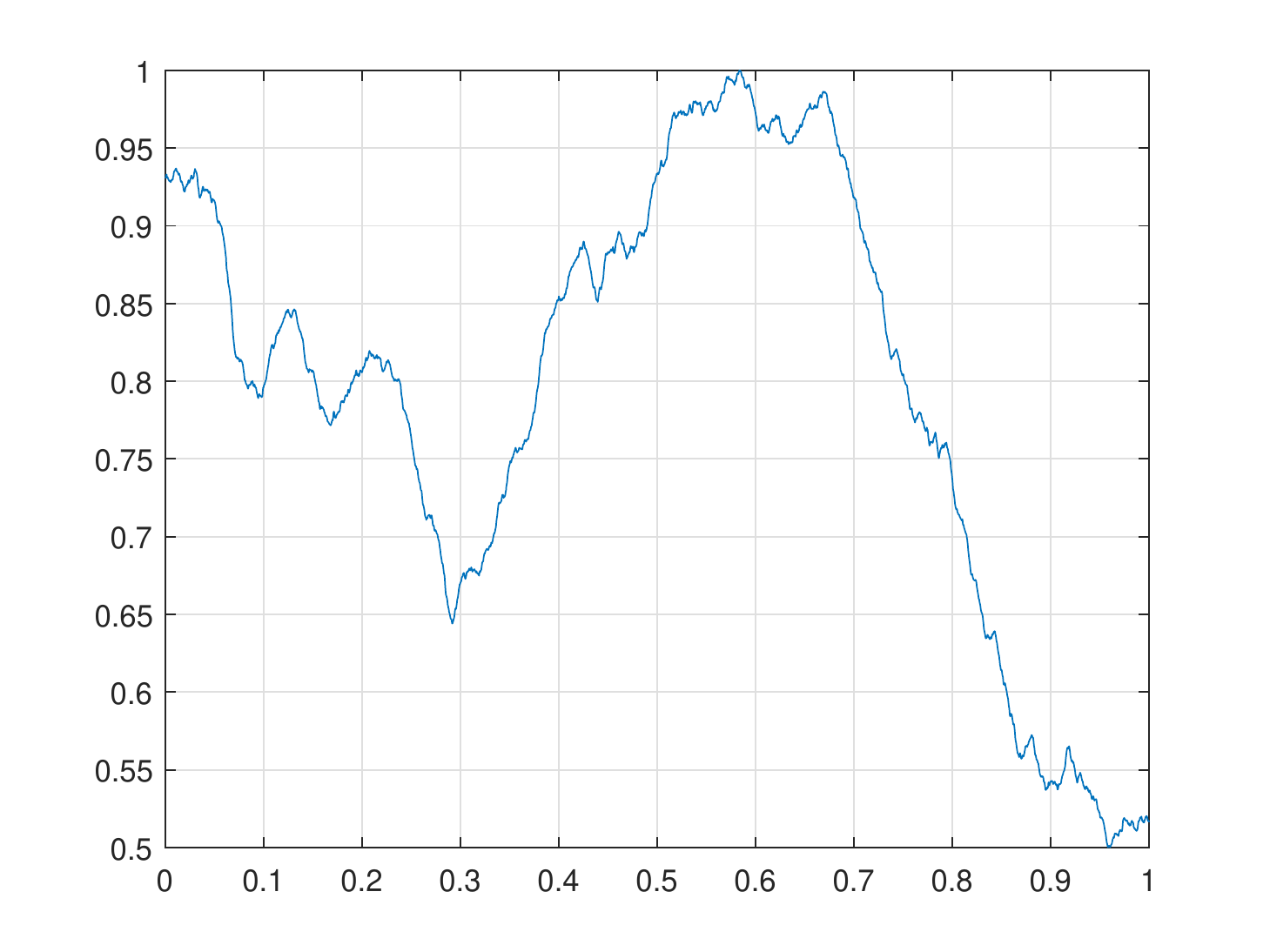}
                         }
    \subfloat[Superposition of realizations of $\widehat{X}^J$ and $\widetilde{X}^J$]{
      \includegraphics[width=0.47\textwidth]{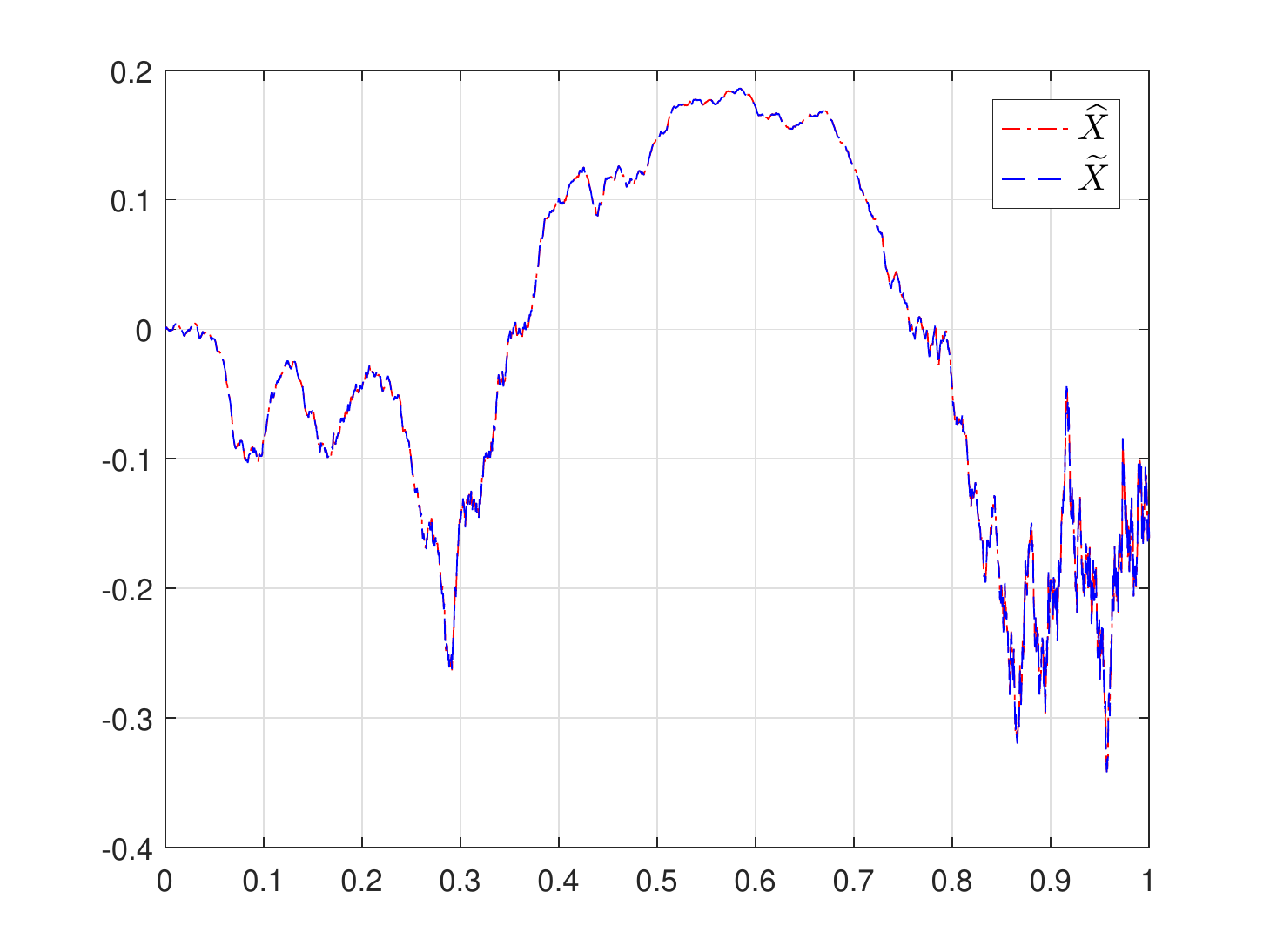}
                         }\\
                         
\end{center}
\end{figure}

\begin{figure}[!ht]
  \begin{center}
    \subfloat[$\widehat{X}^J$]{
      \includegraphics[width=0.47\textwidth]{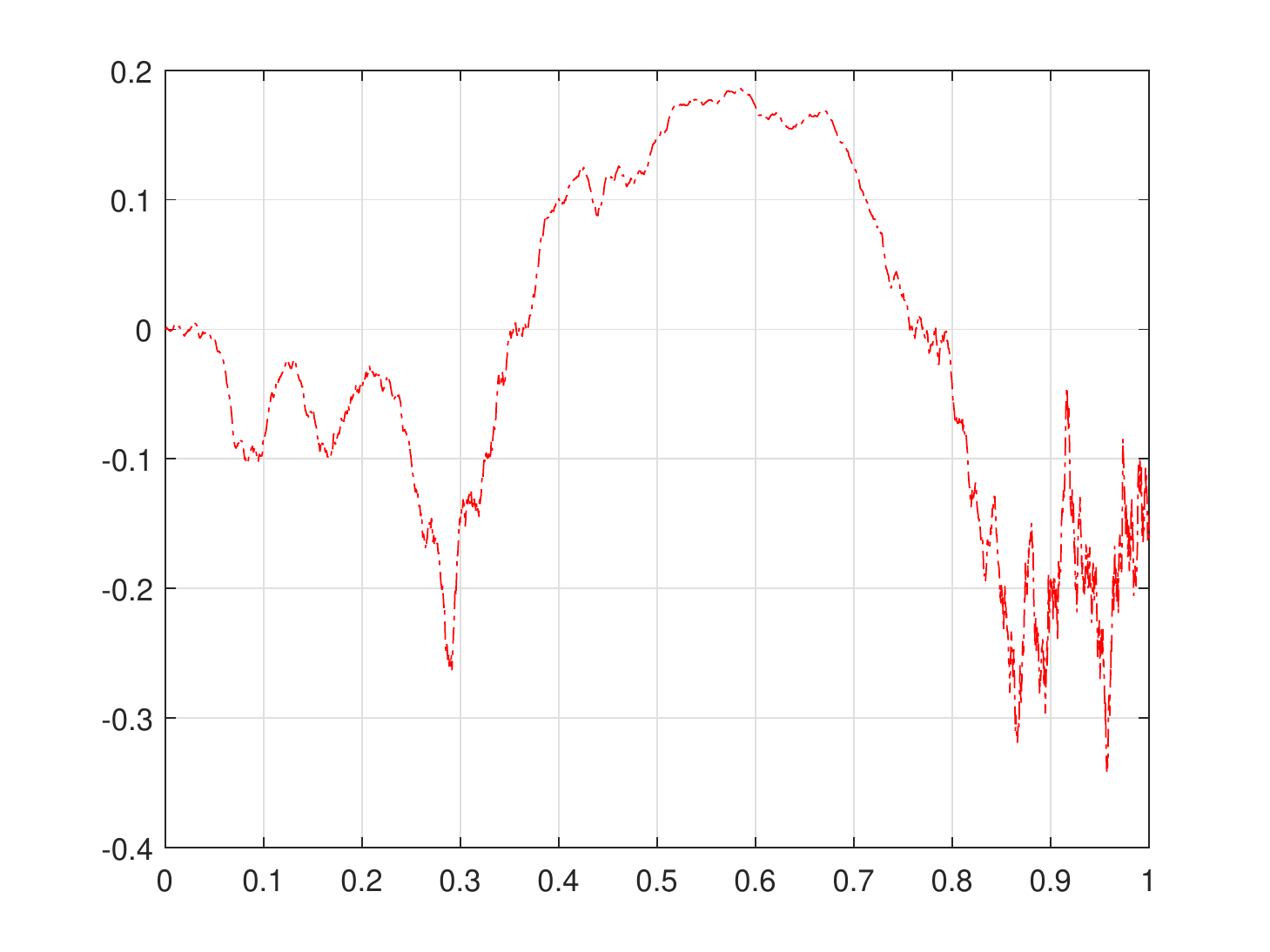}
                         }
    \subfloat[$\widetilde{X}^J$]{
      \includegraphics[width=0.47\textwidth]{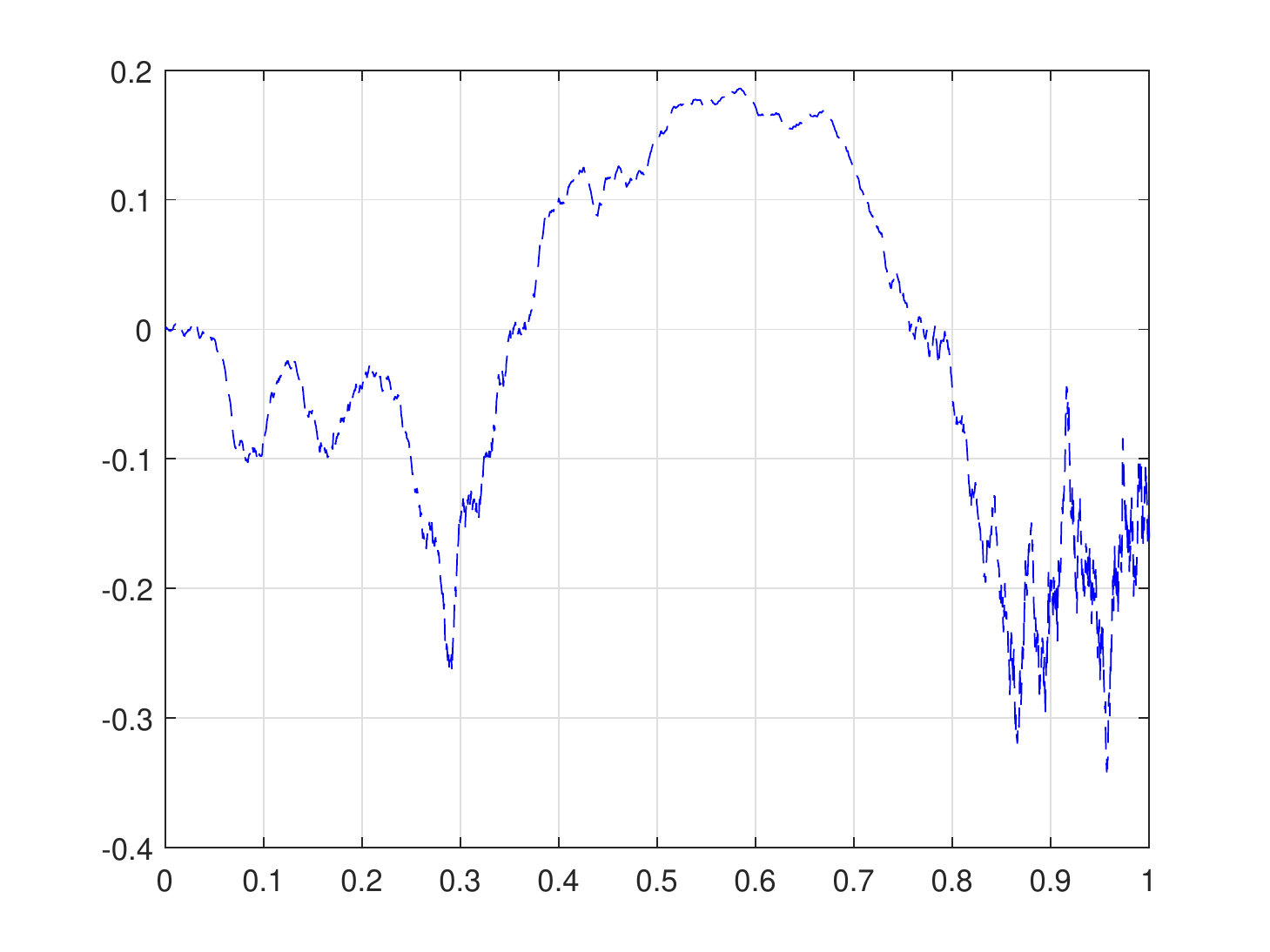}
                         }                       
                         
  \end{center}
\end{figure}

\newpage

\appendix
\section{Proofs of some auxiliary results}
\label{sec:Appen}

\begin{proof}[Proof of Proposition \ref{prop:holdun}]
By combining $(\ref{def_X})$ with the isometry property of It\^o integral, for any real numbers $t',t"$ satisfying $0~\le~t'~\le~t"~\le~1$, one has
\begin{align}
\ESP\left( |X(t")-X(t')|^2 \right)& = \int_0^1 \ESP\left( |K_{t"}(s)-K_{t'}(s)|^2 \right) ds \nonumber \\
& = \int_{t'}^{t"}\ESP\left( (t"-s)^{2A(s)-1} \right) ds  + \int_0^{t'} \ESP\left( \left| (t"-s)^{A(s)-\frac{1}{2}}-(t'-s)^{A(s)-\frac{1}{2}} \right|^2 \right) \label{maj0:prop:holdun}ds\,.
\end{align}
Let us provide a suitable upper bound for each term in the previous sum. The process $\{A(s) : s \in I \}$ being with values in the deterministic interval $[\unda,\ova] \subset (1/2,1)$, it follows that
\begin{equation}\label{maj1:prop:holdun}
\int_{t'}^{t"}\ESP\left( (t"-s)^{2A(s)-1} \right) ds \le \int_{t'}^{t"} (t"-s)^{2\unda-1}ds = \frac{1}{2\unda} (t"-t')^{2\unda}\,. 
\end{equation}
Moreover, using the changes of variables $u=t'-s$ and $v=\frac{u}{t"-t'}$ and standard computations, one gets that 
\begin{align}\label{maj2:prop:holdun}
& \int_0^{t'} \ESP\left( \left| (t"-s)^{A(s)-\frac{1}{2}}-(t'-s)^{A(s)-\frac{1}{2}} \right|^2 \right) ds = \int_0^{t'}\ESP\left( \left| (t"-t'+u)^{A(t'-u)-\frac{1}{2}}-u^{A(t'-u)-\frac{1}{2}}\right|^2 \right)du \nonumber \\
& =\int_{0}^{t'} \ESP \left( |t"-t'|^{2A(t'-u)-1} \left| \left(1+\frac{u}{t"-t'} \right)^{A(t'-u)-\frac{1}{2}} - \left( \frac{u}{t"-t'} \right)^{A(t'-u)-\frac{1}{2}} \right|^2  \right) du \nonumber \\
& \le |t"-t'|^{2\unda-1} \int_{0}^{t'} \ESP \left(  \left| \left(1+\frac{u}{t"-t'} \right)^{A(t'-u)-\frac{1}{2}} - \left( \frac{u}{t"-t'} \right)^{A(t'-u)-\frac{1}{2}} \right|^2  \right) du \nonumber \\
& \le |t"-t'|^{2\unda} \int_{0}^{\frac{t'}{t"-t'}} \ESP \left( \left|(1+v)^{A((1+v)t'-vt")-\frac{1}{2}}-v^{A((1+v)t'-vt")-\frac{1}{2}} \right|^2 \right)dv \nonumber \\
& \le |t"-t'|^{2\unda} \left( \int_{0}^{1} \ESP \left( \left|(1+v)^{A((1+v)t'-vt")-\frac{1}{2}}-v^{A((1+v)t'-vt")-\frac{1}{2}} \right|^2 \right)dv \right. \nonumber \\
& \qquad \qquad \qquad \qquad \left. + \int_{1}^{\ii} \ESP \left( \left|(1+v)^{A((1+v)t'-vt")-\frac{1}{2}}-v^{A((1+v)t'-vt")-\frac{1}{2}} \right|^2 \right)dv  \right) \, .
\end{align}
Next, observe that, one can derive from the inequality $|a^{\beta}-b^{\beta}|\le |a-b|^{\beta}$, where $\beta\in (0,1)$ and all $a,b\in (0,+\infty)$ are arbitrary, that
\begin{equation}\label{maj21:prop:holdun}
 \int_{0}^{1} \ESP \left( \left|(1+v)^{A((1+v)t'-vt")-\frac{1}{2}}-v^{A((1+v)t'-vt")-\frac{1}{2}} \right|^2 \right)dv  \leq \int_0^1 dv = 1\,.
\end{equation}
Also observe that
by applying, for any fixed $v\in [1,+\infty)$, the Mean Value Theorem to the function $x \mapsto (x+v)^{A((1+v)t'-vt")-\frac{1}{2}}$, and by using (\ref{eqn:bornesA}), one obtains that
\begin{equation*}\label{maj22:prop:holdun}
 \left|(1+v)^{A((1+v)t'-vt")-\frac{1}{2}}-v^{A((1+v)t'-vt")-\frac{1}{2}} \right|
 \leq   v^{\overline{a}-\frac{3}{2}}\, , 
\end{equation*}
and consequently that
\begin{equation}\label{maj23:prop:holdun}
 \int_{1}^{\ii} \ESP \left( \left|(1+v)^{A((1+v)t'-vt")-\frac{1}{2}}-v^{A((1+v)t'-vt")-\frac{1}{2}} \right|^2 \right)dv \leq  \int_{1}^{\ii} v^{2\ova-3}dv  = \frac{1}{2-2\ova} \, .
\end{equation}

\noindent Finally, it results from (\ref{maj0:prop:holdun}), (\ref{maj1:prop:holdun}), (\ref{maj2:prop:holdun}), (\ref{maj21:prop:holdun}) and (\ref{maj23:prop:holdun}) that  
\begin{equation*}
\label{eq:A-kol}
\ESP\left( |X(t")-X(t')|^2 \right) \le c_1 |t"-t'|^{2\unda}\,,
\end{equation*}
where $c_1>0$ is a constant not depending on $t',t''$. Thus, one can derive from the Kolmogorov-\v{C}entsov's continuity theorem, given in \cite{karatzas2012brownian}, that the stochastic process $\{X(t) : t\in I\}$ admits a modification 
whose sample paths are, with probability 1, H\"older continuous function of any order $\zeta \in (0,\unda-1/2)$.
\end{proof}

\begin{proof}[Proof of Lemma \ref{lem_accroissK}]
Let us fix $s',s'',t$ three arbitrary real numbers satisfying $0 \leq s' \leq s'' < t \leq 1$. 
It follows from the triangular inequality and (\ref{def_Kt}) that
\begin{align}\label{eq_maj}
\big|K_t(s') - K_t(s'') \big| 
 = & \big|(t-s')^{A(s')-\frac{1}{2}} - (t-s'')^{A(s'')-\frac{1}{2}} \big| \nonumber \\
 \leq &  \big|(t-s')^{A(s')-\frac{1}{2}} - (t-s'')^{A(s')-\frac{1}{2}} \big|   + \big|(t-s'')^{A(s')-\frac{1}{2}} - (t-s'')^{A(s'')-\frac{1}{2}} \big|  \, .
\end{align}
First, let us apply the Mean Value Theorem to the function $x \mapsto (t-x)^{A(s')-\frac{1}{2}}$. We obtain, for some $s_0 \in (s',s'')$, that
\begin{equation}\label{eq_partialx}
 \big|(t-s')^{A(s')-\frac{1}{2}} - (t-s'')^{A(s')-\frac{1}{2}} \big| \leq \left( A(s') - \frac{1}{2} \right) (t-s_0)^{A(s') - \frac{3}{2}} (s''-s')\leq \left(\overline{a}- \frac{1}{2} \right) (t-s'')^{\underline{a} - \frac{3}{2}} (s''-s')  \, .
\end{equation}
Secondly, by applying the Mean Value Theorem to the function $x \mapsto (t-s'')^{x- \frac{1}{2}}$, we obtain, for some $d\in \big [\min\{A(s'), A(s'')\},\max\{A(s'), A(s'')\}\big]\subseteq [\underline{a},\overline{a}]$, that 
\begin{align}\label{eq_partialy}
\big|(t-s'')^{A(s')-\frac{1}{2}} - (t-s'')^{A(s'')-\frac{1}{2}} \big|  & = \left| \log(t-s'') \right| (t-s'')^{d-\frac{1}{2}} \big| A(s')-A(s'') \big| \nonumber \\
& \leq \left| \log(t-s'') \right| (t-s'')^{\underline{a}-\frac{1}{2}} \big| A(s')-A(s'') \big| \, .
\end{align}
 Finally, using the fact that
\begin{equation*}
\sup_{x \in (0,1]} \big| \log(x) \big|  x^{\underline{a}-\frac{1}{2}} < + \infty
\end{equation*}
and using (\ref{eq_maj}), (\ref{eq_partialx}) and (\ref{eq_partialy}), it follows that (\ref{eq:lem_accroissK}) is satisfied.
\end{proof}

\begin{proof}[Proof of Lemma \ref{lem:conv:XJtXt}]
One can derive from (\ref{eqn:XJt3}), (\ref{eqn:XJt}), the triangular and the Cauchy-Schwarz inequalities, and the equality
\begin{equation}\label{eqn:B}
\ESP \left( \big| \Delta B_{J,l} \big|^2\right) = 2^{-J}, 
\end{equation}
that, for all $J \in  \Z_+$ and all $l \in \{0, \dots, 2^J-1\}$, one has 
\begin{equation}\label{maj:lem:conv:XJtXt}
 \ESP\left( \left| X^J(t)-\widetilde{X}^J(t)  \right| \right) \leq \sum_{l=0}^{2^J-1}  \ESP\left( \left| \overline{K}_t^{J,l} - K_t(\delta_{J,l}) \right|  \left| \Delta B_{J,l}  \right| \right) \leq 2^{-\frac{J}{2}} \sum_{l=0}^{2^J-1}  \ESP\left( \left| \overline{K}_t^{J,l} - K_t(\delta_{J,l}) \right|^2    \right)^{\frac{1}{2}}.
\end{equation}
Moreover, (\ref{eqn:mvKt}) implies that
\begin{equation}\label{maj2:lem:conv:XJtXt}
\left| \overline{K}_t^{J,l} - K_t(\delta_{J,l}) \right| \leq 2^J \int_{\delta_{J,l}}^{\delta_{J,l+1}} \big| K_t(s) - K_t(\delta_{J,l}) \big| ds \, .
\end{equation}
Observe that in the case where $l \geq [2^Jt]+1$, $[\cdot]$ being the integer part function, one has that $\delta_{J,l} > t$ and consequently that
\begin{equation*}
\int_{\delta_{J,l}}^{\delta_{J,l+1}} \big| K_t(s) - K_t(\delta_{J,l}) \big| ds = 0 \, ;
\end{equation*}
therefore (\ref{maj2:lem:conv:XJtXt}) implies that
\begin{equation}\label{maj31:lem:conv:XJtXt}
\left| \overline{K}_t^{J,l} - K_t(\delta_{J,l}) \right| =0 \, .
\end{equation}
Also, observe that in the case where $l=[2^Jt]$, using the triangular inequality and (\ref{eq:Lencadr}) one has
\begin{equation}
\int_{\delta_{J,[2^Jt]}}^{\delta_{J,[2^Jt]+1}} \big| K_t(s) - K_t(\delta_{J,[2^Jt]}) \big| ds  \leq 2^{-J+1}\,;
\end{equation}
therefore (\ref{maj2:lem:conv:XJtXt}) entails that 
\begin{equation}\label{maj32:lem:conv:XJtXt}
\left| \overline{K}_t^{J,[2^Jt]} - K_t(\delta_{J,[2^Jt]}) \right| \leq 2 \, .
\end{equation}
Let us now study the last case where $l< [2^{J}t]$. One has $[\delta_{J,l},\delta_{J,l+1}) \subseteq [0,t)$. Thus, one can derive from Lemma~\ref{lem_accroissK} that 
\begin{align}\label{maj34:lem:conv:XJtXt}
\int_{\delta_{J,l}}^{\delta_{J,l+1}} \big| K_t(s) - K_t(\delta_{J,l}) \big| ds 
& \leq   c_0 \left( \int_{\delta_{J,l}}^{\delta_{J,l+1}} \big| A(s) - A(\delta_{J,l}) \big| ds  + \int_{\delta_{J,l}}^{\delta_{J,l+1}} ( t-s)^{\underline{a}-\frac{3}{2}} |s-\delta_{J,l}| ds \right)  \nonumber \\
& \leq  c_0 \left( \int_{\delta_{J,l}}^{\delta_{J,l+1}} \big| A(s) - A(\delta_{J,l}) \big| ds  + 2^{-J }\int_{\delta_{J,l}}^{\delta_{J,l+1}} ( t-s)^{\underline{a}-\frac{3}{2}} ds \right)\, .
\end{align}
Then, using (\ref{maj2:lem:conv:XJtXt}), (\ref{maj34:lem:conv:XJtXt}), the inequality $(x+y)^2\le 2(x^2+y^2)$ where $x,y$ are arbitrary real numbers, the Cauchy-Schwarz inequality and (\ref{hyp:Ahold1}), one gets that 
\begin{align}\label{maj33:lem:conv:XJtXt}
\ESP\left( \left| \overline{K}_t^{J,l} - K_t(\delta_{J,l}) \right|^2    \right) 
& \leq   2^{2J}\ESP \left( \left( \int_{\delta_{J,l}}^{\delta_{J,l+1}} \big| K_t(s) - K_t(\delta_{J,l}) \big| ds \right)^2 \right) \nonumber \\ 
& \leq 2^{2J} c_0^2 \ESP \left(  \left( \int_{\delta_{J,l}}^{\delta_{J,l+1}} \big| A(s) - A(\delta_{J,l}) \big| ds  + 2^{-J }\int_{\delta_{J,l}}^{\delta_{J,l+1}} ( t-s)^{\underline{a}-\frac{3}{2}} ds \right)^2 \right) \nonumber \\
& \leq 2^{2J+1} c_0^2 \left(  \ESP \left( \left( \int_{\delta_{J,l}}^{\delta_{J,l+1}}  \big| A(s) - A(\delta_{J,l}) \big|  ds \right)^2 \right)  + 2^{-2J }\left(\int_{\delta_{J,l}}^{\delta_{J,l+1}} ( t-s)^{\underline{a}-\frac{3}{2}} ds \right)^2 \right) \nonumber \\
& \leq 2^{2J+1} c_0^2 \left(  2^{-J} \int_{\delta_{J,l}}^{\delta_{J,l+1}} \ESP \left( \big| A(s) - A(\delta_{J,l}) \big|^2 \right) ds  + 2^{-2J }\left(\int_{\delta_{J,l}}^{\delta_{J,l+1}} ( t-s)^{\underline{a}-\frac{3}{2}} ds \right)^2 \right) \nonumber \\
& \leq 2 c_0^2 \left(  c 2^{-2\rho J} + \left(\int_{\delta_{J,l}}^{\delta_{J,l+1}} ( t-s)^{\underline{a}-\frac{3}{2}} ds \right)^2 \right) \, .
\end{align}

Finally, it follows from (\ref{maj:lem:conv:XJtXt}), (\ref{maj31:lem:conv:XJtXt}), (\ref{maj32:lem:conv:XJtXt}) and (\ref{maj33:lem:conv:XJtXt}) that
\begin{align*}
\ESP\left( \left| X^J(t)-\widetilde{X}^J(t)  \right| \right)  & \leq 2^{-\frac{J}{2}} \sum_{l=0}^{2^J-1}  \ESP\left( \left| \overline{K}_t^{J,l} - K_t(\delta_{J,l}) \right|^2    \right)^{\frac{1}{2}} \nonumber \\
 & \le 2^{-\frac{J}{2}} \ESP\left( \left| \overline{K}_t^{J,[2^J t]} - K_t(\delta_{J,[2^J t]}) \right|^2    \right)^{\frac{1}{2}} + 2^{-\frac{J}{2}} \sum_{l=0}^{[2^Jt]-1} \ESP\left( \left| \overline{K}_t^{J,l} - K_t(\delta_{J,l}) \right|^2    \right)^{\frac{1}{2}} \nonumber \\
 & \le 2^{-\frac{J}{2}+1} +  2^{-\frac{J-1}{2}}  c_0 \sum_{l=0}^{[2^Jt]-1} \left( \sqrt{c} 2^{-\rho J} +  \int_{\delta_{J,l}}^{\delta_{J,l+1}} ( t-s)^{\underline{a}-\frac{3}{2}} ds  \right)   \nonumber\\
 & \le 2^{-\frac{J}{2}+1} +  2^{-\frac{J-1}{2}}  c_0  \left( \sqrt{c}\sum_{l=0}^{[2^Jt]-1}  2^{-\rho J} +   \int_{0}^{t} ( t-s)^{\underline{a}-\frac{3}{2}} ds \right) \nonumber \\
& \le 2^{-\frac{J}{2}+1} + 2^{-\frac{J-1}{2}} c_0 \left(\sqrt{c}  2^{-\rho J} 2^J + \frac{t^{\underline{a}-\frac{1}{2}}}{\underline{a}-\frac{1}{2}} \right)
 \le c_2 2^{-J\left(\rho-\frac{1}{2}\right)} \, ,
 \end{align*}
 where $c_2:= 2 +  \sqrt{2}c_0 \left(\sqrt{c} + \frac{1}{\underline{a}-\frac{1}{2}} \right)$. Then the last inequality and the fact that $\rho>1/2$ show that (\ref{eqn:conv:XJtXt}) is satisfied.
\end{proof}

\section*{Acknowledgements}
This work has been partially supported by the Labex CEMPI (ANR-11-LABX-0007-01), the F.R.S.-FNRS and the GDR 3475 (Analyse Multifractale).

\end{document}